\documentclass[10pt, amsmath, amssymb]{amsart}

\usepackage{amsmath}
\usepackage{amsthm}
\usepackage{amssymb}
\usepackage{graphics}
\usepackage{bm}

\usepackage{bbold}

\newcommand{\be}{\begin{equation}}
\newcommand{\ee}{\end{equation}}
\newcommand{\ba}{\begin{eqnarray}}
\newcommand{\ea}{\end{eqnarray}}
\newcommand{\bi}{\begin{itemize}}
\newcommand{\ei}{\end{itemize}}
\newcommand{\bn}{\begin{enumerate}}
\newcommand{\en}{\end{enumerate}}
\newcommand{\bp}{\begin{proof}}
\newcommand{\ep}{\end{proof}}

\newcommand{\mr}{\ensuremath{\mathrm}}

\newcommand{\mc}{\ensuremath{\mathcal}}
\newcommand{\mf}{\ensuremath{\mathfrak}}
\newcommand{\ov}{\ensuremath{\overline}}
\newcommand{\sm}{\ensuremath{\setminus}}

\newcommand{\Om}{\ensuremath{\Omega}}

\newcommand{\La}{\ensuremath{\Lambda }}
\newcommand{\la}{\ensuremath{\lambda }}

\renewcommand{\bm}{\ensuremath{\mathbb }}
\newcommand{\tr}{\ensuremath{\mathrm{tr} }}

\newcommand{\ip}[2]{\ensuremath{\langle {#1} , {#2} \rangle}}

\newcommand{\dom}[1]{\ensuremath{\mathrm{Dom} ({#1}) }}
\renewcommand{\dim}[1]{\ensuremath{\mathrm{dim} \left( {#1} \right) }}

\newcommand{\ran}[1]{\ensuremath{\mathrm{Ran} ({#1}) }}
\renewcommand{\ker}[1]{\ensuremath{\mathrm{Ker} ({#1}) }}

\newcommand{\im}[1]{\ensuremath{\mathrm{Im} \left( {#1} \right) }}
\newcommand{\re}[1]{\ensuremath{\mathrm{Re} \left( {#1} \right) }}

\numberwithin{equation}{section}

\newtheorem{thm}[subsubsection]{Theorem}

\newtheorem{lemming}[subsubsection]{Lemma}
\newtheorem{prop}[subsubsection]{Proposition}

-.9truein

\advance\voffset by -.1truein 

\begin{document}

\bibliographystyle{unsrt}

\title[Unitary perturbations of compressed shifts]{Unitary perturbations of compressed n-dimensional shifts}

\author{R.T.W. Martin}

\address{Department of Mathematics and Applied Mathematics \\ University of Cape Town\\
Cape Town, South Africa \\
phone: +27 21 650 5734 \\ fax: +27 21 650 2334}

\email{rtwmartin@gmail.com}

\begin{abstract}

    Given a purely contractive matrix-valued analytic function $\Theta$ on the unit disc $\bm{D}$,
we study the $\mc{U} (n)$-parameter family of unitary perturbations of the operator $Z_\Theta$
of multiplication by $z$ in the Hilbert space $L^2 _\Theta$ of $n-$component vector-valued functions
on the unit circle $\bm{T}$ which are square integrable with respect to the matrix-valued measure $\Om _\Theta$
determined uniquely by $\Theta$ and the matrix-valued Herglotz representation theorem.

    In the case where $\Theta $ is an extreme point of the unit ball of bounded $\bm{M} _n$-valued
functions we verify that the $\mc{U} (n)$-parameter family of unitary perturbations of $Z_\Theta ^*$
is unitarily equivalent to a $\mc{U} (n)$-parameter family of unitary perturbations of $X _\Theta$,
the restriction of the backwards shift in $H^2 _n (\bm{D})$, the Hardy space of $\bm{C} ^n$ valued functions on
the unit disc, to $K^2 _\Theta$, the de Branges-Rovnyak space constructed using $\Theta$.
These perturbations are higher dimensional analogues of the unitary perturbations introduced
by D.N. Clark in the case where $\Theta$ is a scalar-valued ($n=1$) inner function, and studied by E. Fricain
in the case where $\Theta$ is scalar-valued and an extreme point of the unit ball of $H^\infty (\bm{D})$.

    A matrix-valued disintegration theorem for the Aleksandrov-Clark measures associated with matrix-valued contractive
analytic functions $\Theta$ is obtained as a consequence of the Weyl integration formula for $\mc{U}(n)$ applied to
the family of unitary perturbations of $Z _\Theta$. This disintegration formula generalizes a recent result of S. Elliott
to arbitrary matrix-valued contractive analytic functions. Following results of Clark and Fricain in the
scalar case, a necessary and sufficient condition on $\Theta$ for $K^2 _\Theta$ to contain a total orthogonal set of
point evaluation or reproducing kernel vectors is provided.

\vspace{5mm}   \noindent {\it Key words and phrases}:
Hardy space, model subspaces, Aleksandrov disintegration theorem, Clark's unitary peturbations, Aleksandrov-Clark measures,
matrix-analytic functions, symmetric/isometric linear transformations

\vspace{3mm}

\end{abstract}

\maketitle

\section{Introduction}

    Let $\Theta$ be an $\bm{M} _n$-valued contractive analytic function on $\bm{D}$, the unit disc in the complex plane $\bm{C}$. Here
$\bm{M} _n$ denotes the $n \times n$ matrices with entries in $\bm{C}$. Recall \emph{cf.}
\cite[Proposition V.2.1]{Foias}, that $\Theta$ can be block-diagonalized as $\Theta = \Theta _0 \oplus \Theta _1$ where $\Theta _0$ is a unitary constant and
$\Theta _1$ is purely contractive, \emph{i.e.} $\| \Theta _1 (0) \| < 1$.  We will assume throughout that $\Theta $ is purely contractive.
For such a function it follows easily that $\| \Theta (z) \| <1 $ for all $z \in \bm{D}$. Recall that the function $\Theta$ is
said to be inner if $\Theta (\zeta )$, $\zeta \in \bm{T}$, is unitary a.e. with respect to Lebesgue measure on the unit circle $\bm{T}$ (here
$\Theta (\zeta)$ is the non-tangential limit of $\Theta (z)$ for $z$ approaching $\zeta$ non-tangentially in $\bm{D}$).

    Given any $A \in \ov{(\bm{M} _n) _1 }$, the closed unit ball of $\bm{M} _n$, let $\Theta _A := \Theta A^*$ and
define \be B _{\Theta _A} (z) := \frac{1 + \Theta (z) A^*}{1 -\Theta (z) A^*} .\ee This is clearly analytic on $\bm{D}$ since $\| A \|
\leq 1$ and $\| \Theta (z) \| < 1$ for all $z \in \bm{D}$.
It is straightforward to calculate that \be \re{B _{\Theta _A} (z)} := \frac{1}{2} (B _{\Theta _A} (z) + B _{\Theta _A} (z) ^*) = (1- \Theta (z) A^*) ^{-1} (1- \Theta (z) A^* A \Theta (z) ^* )
(1-A \Theta (z) ^* )^{-1} .\ee This is clearly positive so that by the matrix-valued Herglotz theorem \cite[Theorem 3]{Elliott}, it follows that for each such $A$ there is a unique positive $\bm{M} _n$ valued measure $\Om _{\Theta _A}$ on $\bm{T}$ such that
\be \re{B _{\Theta _A} (z)} = \re{\int _\bm{T} \frac{\zeta +z}{\zeta -z} \Om _{\Theta _A} (d\zeta) } .\ee The imaginary part of $B _{\Theta _A} (z)$
is \be \im{B _{\Theta _A} (z)} := \frac{1}{2i} ( B _{\Theta _A} (z) + B _{\Theta _A} (z) ^*) = -i (1-\Theta (z) A^*) ^{-1} (\Theta (z) A^*
-A \Theta (z) ^* ) (1 - A \Theta (z) ^*) ^{-1} .\ee It is then straightforward to calculate that
\be B _{\Theta _A} (z) = \int _\bm{T} \frac{\zeta +z}{\zeta -z} \Om _{\Theta _A} (d\zeta) + i \im{B _{\Theta _A} (0)} \label{eq:glotz}.\ee

    In the case where $A$ is unitary, the measures $\Om _{\Theta _A}$ are the matrix-valued Aleksandrov-Clark measures introduced in \cite{Elliott}.
In the case where $\Theta $ is a scalar-valued and $A \in \bm{T}$, these are the usual Aleksandrov-Clark measures, first introduced
in \cite{Clark-perturb}, and studied since by many authors. We will sometimes write $\Om _A$ and $B_A$ in place of $\Om _{\Theta _A}$
and $B_{\Theta _A}$ respectively when there is no chance of confusion.

     Given a contractive analytic function $\Theta$, let $L^2 _\Theta (\bm{T})$ or simply $L^2 _\Theta$ denote the
Hilbert space of $\bm{C} ^n$-valued functions on $\bm{T}$ which are square integrable with respect to the matrix-valued
measure $\Om _\Theta := \Om _{\Theta _{\bm{1}_n}}$.  Explicitly, let $\{e_i \} _{i=1} ^n$, $n = \mr{rank} (\Theta )$ be
a fixed orthonormal basis for $\bm{C} ^n$, $\Om _\Theta ( I )_{ij}$ the matrix entries of $\Om _\Theta ( I )$ with respect to this
basis ($I \subset \bm{T}$ is some fixed Borel set). The Hilbert space $L^2 _\Theta$ contains a copy of $\bm{C} ^n$.
It will be convenient to denote the embedding of $\bm{C} ^n$ into $L^2 _\Theta $ by $V_n$. Let $b _i ^- := V_n e_i$,
the $b_i ^-$ are the constant functions $b_i ^- (\zeta ) = e_i$, $\zeta \in \bm{T}$.
Elements $f,g \in L^2 _\Theta$, will be viewed as column
vectors of functions with entries $f_i (\zeta) := (f (\zeta) , b _i ^- (\zeta) )$, where $(\cdot , \cdot)$ denotes the inner product in $\bm{C} ^n$.
Then the inner product in $L^2 _\Theta$ is given by the formula
\be \left( f , g \right) _\Theta := \int _\bm{T} \left( \Om _\Theta (d\zeta ) f (\zeta ) , g(\zeta ) \right) = \sum _{i,j =1} ^n \int _\bm{T}
\ov{g _i (\zeta )} \Om _\Theta (d\zeta ) _{ij} f _j (\zeta ) .\ee
 Let $Z _\Theta$ denote the operator of multiplication by the independent variable $\zeta $ in this space. Let $\mf{D} _+
:= \bm{C} \{ b _i ^+ \} _{i=1} ^n$, the subspace spanned by the vectors $b _i ^+ (\zeta ) := 1/ \zeta  b _i ^- (\zeta )$.
Here $\bm{C} \{ b_i ^+ \}$ denotes the linear span of the set $\{ b_i ^+ \}$.
Then let $\mf{D} _- := Z_\Theta \mf{D} _+ = \bm{C} \{ b _i ^- \}$. Let $P_\pm$ denote the projectors onto $\mf{D} _\pm$.

In what follows, we assume that $\Theta (0) = 0$
so that $\Om _\Theta (\bm{T}) = \bm{1} _n$ and the $b _i ^\pm$ are orthonormal basis vectors for $\mf{D} _\pm$.
Given any $A \in \ov{ (\bm{M} _n) _1 } $, we will identify $A$ with the operator $\hat{A} \in \mc{B} (L^2 _\Theta)$ defined by
\be \hat{A} := \sum _{i,j =1} ^n ( \cdot , b _i ^- ) _\Theta A_{ij} b _j ^- = \left( (\cdot, b _1 ^- )_\Theta,
 ..., (\cdot , b _n ^- )_\Theta \right)  A  \left( \begin{array}{c} b _1 ^- \\ \vdots \\ b _n ^- \end{array} \right) . \ee We will
identify $\hat{A}$ with $A$ and simply write $A$ for $\hat{A}$ from now on.  For each such $A$ define $Z _\Theta (A) := Z_\Theta  + P_- (A -\bm{1} _n ) P_- Z _\Theta$, a perturbation of
$Z_\Theta$. To simplify notation, we will sometimes write $Z(A)$ in place of $Z_\Theta (A)$ when the choice of $\Theta $ is clear. If $A \in \mc{U} (n)$ then $Z_\Theta (A)$ is unitary, and $Z _\Theta (\bm{1}) =Z_\Theta$. Here $\mc{U} (n)$ denotes the group
of unitary $n \times n$ matrices. The family of unitary operators $Z_\Theta (U); \ \ U \in \mc{U} (n)$ can be seen as the family of unitary extensions of the simple isometric linear transformation $Z ' _\Theta := Z_\Theta (0) | _{L^2 _\Theta \ominus \mf{D} _+}$. This will be discussed in greater detail in Section \ref{section:symrep}.

    This paper will now proceed as follows. Consider $\La _{\Theta (U)} (I) := \chi _I ( Z_{\Theta} (U) )$,
where $U \in \mc{U} (n)$, $I$ is a Borel subset of $\bm{T}$, $\chi _I$ is the characteristic function of $I \subset \bm{T}$,
and $\chi _I (Z_\Theta (U))$ is a spectral projection defined using the Borel functional calculus for the unitary operator $Z_\Theta (U)$. In the
next section we will prove that $\Om _{\Theta _U} ( I ) = [ (\Om _{\Theta _U})  _{ij} ( I) ] = [ ( \La_{\Theta (U)} (I) b _i ^- ,
b _j ^- ) _\Theta ]$. With this identification and a straightforward application of the Weyl integration formula for the Lie
group $\mc{U} (n)$ a matrix-version of Aleksandrov's disintegration theorem for arbitrary $\bm{M} _n$-valued purely contractive analytic functions on $\bm{D}$ satisfying $\Theta (0) =0$ will be established. This will extend the main result of Elliott \cite[Theorem 15]{Elliott} which establishes the disintegration theorem for $\Theta$ which are the product of a scalar function in $(H^\infty (\bm{D}) )_1$ with an inner matrix function satisfying $\Theta (0) =0$.

In Section \ref{section:cauchyint}, the Cauchy integral representation for the de Branges-Rovnyak space $K^2 _\Theta$, associated with $\Theta$ as presented in \cite[Chapter III]{Sarason-dB}, is adapted to the case where $\Theta$ is $\bm{M} _n$-valued (we refer the reader to this section for the formal definition of $K^2 _\Theta$). In direct analogy with the scalar case it is shown that there is a unitary transformation $V_\Theta $ of $H^2 _\Theta $, the closure of the polynomials in $L^2 _\Theta$ onto the de Branges-Rovnyak space $K^2 _\Theta $ which takes $Z_\Theta ^*$ onto a rank $n$
perturbation of $X_\Theta := S^*| _{K^2 _\Theta}$, the restriction of the backwards shift $S^*$ to $K^2 _\Theta$. We will then verify that, as in the case where $\Theta$ is
scalar, $H^2 _\Theta = L^2 _\Theta$ if and only if $\Theta $ is an extreme point of the unit ball of $H^\infty _{\bm{M} _n} (\bm{D})$, the Hardy space of $\bm{M} _n$-valued analytic functions on $\bm{D}$ whose supremum norms on circles of radius $0\leq r <1$ are uniformly bounded.  We will further check that $\Theta$ is an extreme point if and only if the trace of $\ln (\bm{1} - | \Theta | ) $ fails to be Lebesgue integrable on $\bm{T}$. In the case that $\Theta $ is an extreme point
of $\left( H^\infty _{\bm{M} _n} (\bm{D}) \right) _1$, the open unit ball of $H^\infty _{\bm{M} _n} (\bm{D} )$ and $\Theta (0) = 0$, it will be verified that the image of  $Z_\Theta (0) ^*$ under the unitary transformation $V_\Theta$ is $X _\Theta$, and that the image of the family of unitary perturbations $Z_\Theta (U) ^*$ under this transformation is a $\mc{U} (n)$ family of unitary perturbations of the restricted backwards shift $X_\Theta$. In the case where $n=1$, this family is the $\mc{U} (1)$ family of
unitary perturbations introduced by D.N. Clark in \cite{Clark-perturb} for $\Theta $ inner.

    Given $\Theta \in \left( H^\infty _{\bm{M} _n } (\bm{D}) \right) _1$ and $z, w \in \bm{D}$, consider
the reproducing kernel matrix function $\Delta _w (z):= \frac{\bm{1} - \Theta (z) \Theta ^* (w) }{1-z \ov{w}}$. Then for any $\vec{x} \in \bm{C} ^n$,
$ \delta ^{\vec{x}} _z := \Delta _z  \vec{x}$ belongs to $K^2 _\Theta$ and is such that for any $f \in K^2 _\Theta$,
$\ip{f}{\delta^{\vec{x}} _z} _\Theta = (f (z) , \vec{x} ) $. Here $\ip{\cdot}{\cdot} _\Theta$ denotes the inner product
in $K^2 _\Theta$. We call the functions $\delta^{\vec{x}} _z$ the reproducing kernel functions or the point
evaluation functions at the point $z \in \bm{D}$. When $\Theta $ is scalar valued, \cite{Fricain} (see also \cite{Clark-perturb} for the inner case)
provides necessary and sufficient conditions on $\Theta$ for $K^2 _\Theta$ to have a total orthogonal set of point evaluation functions. In Section 4,
it is shown that these results have a direct and straightforward generalization to the case where $\Theta$ is matrix valued, and the spectrum
of the unitary perturbations $Z_\Theta (U)$ are calculated. In the process of achieving this, the analogues of several results on Carath$\mr{\acute{e}}$odory angular derivatives for contractive analytic functions on $\bm{D}$ as presented in \cite[Chapter VI]{Sarason-dB} are verified for matrix-valued $\Theta$.

Finally in Section \ref{section:symrep} we consider the isometric linear transformation $Z_\Theta ' := Z_\Theta (0) | _{\mc{D} _+ ^\perp}$. This is a simple
isometric linear transformation with deficiency indices $(n,n)$ and Lifschitz characteristic function equal to $\Theta$ \cite{Lifschitz2}. Let
$\mu (z) := \frac{z-i}{z+i} ; \ \ \mu : \bm{U} \rightarrow \bm{D}$ where $\bm{U}$ denotes the open upper half-plane. Then $\mu ^{-1} (z) = i \frac{1+z}{1-z}$. Using the theory of Lifschitz we determine when the inverse Cayley transform $\mu ^{-1} (Z_\Theta ')$ of $Z_\Theta '$ is a densely defined symmetric operator.
If $\Theta$ is inner, the canonical unitary transformation that takes $L^2 _\Theta = H^2 _\Theta$
to $K^2 _\Phi$ where $\Phi := \Theta \circ \mu $ is a contractive analytic function on $\bm{U}$ and $K^2 _\Phi = H^2 _n (\bm{U}) \ominus \Phi H^2 _n (\bm{U})$, maps $\mu ^{-1} (Z ' _\Theta ) $ onto $M_\Phi$, the symmetric operator of multiplication by $z$ in $K^2 _\Phi$. We verify that, as in the
scalar (n=1) case, $K^2 _\Phi$ has a $\mc{U} (n)$-parameter family of total orthogonal sets of point evaluation vectors $\{ \delta _{\la _j (U)} ^{\vec{x} _j (U) } \} _{j \in \bm{Z};  \ U \in \mc{U} (n)} $, such that the sequences $(\la _j (U)) \subset \bm{R}$ have no finite accumulation point (It will be shown in Section \ref{section:rkhssamp} that $(\la _j (U) ) _{j\in \bm{Z} }$ is necessarily a sequence of real values) if and only if $\Theta$ is analytic on some open neighbourhood of any given $x \in \bm{R}$. Here $K^2 _\Phi \subset H^2 _n (\bm{U}) \subset L^2 _n (\bm{R})$, where $L^2 _n (\bm{R})$ is the Hilbert space of $\bm{C}^n$ valued functions on $\bm{R}$ which are square integrable with respect to Lebesgue measure. This provides a class of vector-valued reproducing kernel Hilbert spaces of functions on $\bm{R}$ which have total orthogonal sets of point evaluation vectors.

Such reproducing kernel Hilbert spaces have the special property that their elements are perfectly reconstructible from the values they take on certain discrete sets of points. Indeed, suppose that $\mc{H}$ is a RKHS of $\bm{C} ^n$-valued functions on a set $X \subset \bm{C}$, \emph{i.e.} for any $\vec{y} \in \bm{C} ^n$ and any $x \in X$, the linear functional which evaluates
an element $f \in \mc{H}$ at $x$ and takes its inner product with $\vec{y}$ is bounded. By the Riesz representation theorem, for each $\vec{y} \in \bm{C} ^n$
and $x \in X$, there is then a `point evaluation vector' $\delta _x ^{\vec{y}} \in \mc{H}$ such that $\ip{ f}{\delta _x ^{\vec{y}}} = \left( f(x) , \vec{y} \right)$. Here $\ip{\cdot}{\cdot}$ is the inner product on $\mc{H}$ and $(\cdot , \cdot)$ is, as before, the inner product in $\bm{C} ^n$. If $\mc{H}$ has a total orthogonal set of point evaluation vectors $\{ \delta _{x_j} ^{\vec{y} _j} \}$, then it follows that for any $f \in \mc{H}$, \be f = \sum _j \ip{f}{\delta _{x_j } ^{\vec{y} _j} } \frac{\delta _{x_j} ^{\vec{y} _j}}{\| \delta _{x_j} ^{\vec{y} _j} \| ^2 } = \sum _j \left( f(x_j) , \vec{y} _j \right) \frac{\delta _{x_j} ^{\vec{y} _j} }{\left( \delta _{x_j} ^{\vec{y} _j}
( x_j ) , \vec{y} _j \right)} .\ee This shows that any element $f \in \mc{H}$ can be perfectly reconstructed from the values
$\{ \left( f(x_j ) , \vec{y} _j \right) \}$, the $\bm{C} ^n$-inner products of its values taken on the set of points $\{x _j \} \subset X$ with
the vectors $ \vec{y} _j  \in \bm{C} ^n$.

\section{The Matrix-valued Aleksandrov disintegration theorem}

\label{section:ACdis}

    Throughout this section we will assume that $\Theta (0) = 0$.

\subsection{Identification of the matrix-valued Aleksandrov-Clark measures}
\label{subsection:ACmeas}

The purpose of this subsection is to establish two key facts needed for the proof of the disintegration theorem. First it will be shown that the Aleksandrov-Clark measures $\Om _{\Theta _U}$ associated with $\Theta$ and unitary $U \in \mc{U} (n)$ are such that $(\Om _{\Theta _U} ( I ) )_{ij} = ( \chi _I (Z_\Theta (U) ) b ^+ _i , b ^+ _j ) _\Theta$ where $b ^+ _i (z) = 1/z b _i ^-$ for $z \in \bm{T}$, and $b _i ^-$ are the constant co-ordinate functions. Recall here that $\chi _I$ is the characteristic function of the Borel subset $I \subset \bm{T}$. In fact, we will establish something stronger than this. Given any $A \in (\bm{M} _n )  _1$ the operator $Z_\Theta (A)$ is a completely non-unitary contraction since if $\| A \| <1$, any unitary restriction of $Z_\Theta (A)$ would have to be a unitary restriction of $Z_\Theta $ to a subspace orthogonal to $\mf{D} _+ = \bm{C} \{ 1/z b _i ^- \}$. This is not possible. If $Z_\Theta$ has a unitary restriction to a subspace $S$ which is orthogonal to $\mf{D} _+$, then $S$ is reducing for $Z_\Theta$
and $Z_\Theta ^k S = S $ for all $k \in \bm{Z}$. Since $Z_\Theta$ is unitary it would then follow that $S$ is orthogonal to $\bigvee _{k\in \bm{Z}}
\{ z^k b_i ^- \} _{i=1} ^n$ which is dense in $L^2 _\Theta$ so that $S = \{ 0 \}$. Let $U_A$ acting on $\mc{K} _A \supset L^2 _\Theta$ be the minimal unitary dilation of $Z_\Theta (A)$, and let $P_A$ be the orthogonal projection of $\mc{K} _A$ onto $L^2 _\Theta$. We will show
that the positive matrix-valued measure $\La _A ( I ) := P_A \chi _I ( U_A) P_A$, $I \in \mr{Bor} (\bm{T})$ is such that $\La _A ( I ) _{ij} = (\La_A ( I ) e_i , e_j ) = ( \Om _{\Theta _A} ( I ) b _i ^+, b ^+ _j ) _\Theta$. Here $\mr{Bor} (\bm{T})$ denotes the Borel subsets of $\bm{T}$.

 Secondly we will show that $\Om _0 = m $ where $m$ denotes $\bm{M} _n$-valued
normalized Lebesgue measure on $\bm{T}$, \emph{i.e.} $m ( I ) _{ij} = \mu ( I ) \delta _{ij}$, and $\mu$ is normalized Lebesgue measure on $\bm{T}$,
and $\delta _{ij}$ is the Kronecker delta.

    To identify the matrix valued Aleksandrov-Clark measures $\Om _{\Theta _A}$ with the spectral measures associated with the perturbations $Z_\Theta (A)$,
for any $A \in \ov{(\bm{M} _n) _1}$, we will apply the following Proposition taken from \cite[Proposition 14]{Elliott}.  Although the statement of the
proposition in \cite{Elliott} assumes that $A$ is unitary, the proof for general $A$ is identical.

\begin{prop}{ (S. Elliott) }
    Let $\Theta : \bm{D} \rightarrow \bm{M} _n$ be analytic and purely contractive with $\Theta (0) =0$. Then for any $A \in \ov{(\bm{M} _n) _1}$,
\be \int _\bm{T} \zeta ^n \Om _{\Theta _A} (d\zeta ) = \left\{ \begin{array}{cc}
\sum _{k=1} ^n \int _\bm{T} \zeta ^{-n} (A\Theta (\zeta ) ^* ) ^k  m (d\zeta) & n \geq 1 \\
\sum _{k=1} ^{|n|} \int _\bm{T} \zeta ^n (\Theta (\zeta ) A^*) ^k m (d\zeta) & n \leq -1 \\
\bm{1} & n =0 \end{array} \right. \ee \label{prop:Elliott}
\end{prop}

\begin{prop}
    Let $\Theta \in \left( H^\infty _{\bm{M} _n} (\bm{D}) \right) _1$ be a purely contractive
analytic function with $\Theta (0) = 0$. Then $( \Om _{\Theta _A} (\cdot)  e_i , e_j ) = (\La_A (\cdot) b _i ^-,
b _j ^- ) _\Theta$ where $\La_A$ is the positive $\bm{M} _n$-valued measure associated with $Z_\Theta (A)$. \label{prop:acmeasure}
\end{prop}

\subsubsection{Notation} Here $(e_i )  _{i=1} ^n$ will be an orthonormal basis of $\bm{C} ^n$ that is fixed throughout this paper. As in the
introduction $V_n : \bm{C} ^n \rightarrow L^2 _\Theta$ is an isometry defined by $V_n e_i = b _i ^-$, where the $b _i ^-$ form an
orthonormal basis for $\mf{D} _-$, the copy of $\bm{C} ^n$ in $L^2 _\Theta$. The above proposition can be stated more succinctly as
$V_n ^* P_- \La _A  P_- V_n = \Om _{\Theta _A} $.

    Recall that if $A=U$ is unitary then $\La _U ( I) = \chi _I (Z_\Theta (A))$ is a projection for any $I \subset \mr{Bor} (\bm{T})$,
the Borel subsets of $\bm{T}$.
To simplify the presentation of the proof, we first establish a few lemmas. Consider the power series for $\Theta$,
$\Theta (z) := \sum _{k=1} ^\infty c_k z^k$, $c_k \in \bm{M} _n$ (recall we assume that $\Theta (0) = 0$). Let $l_j (A)$ denote
the $j^{\mathrm{th}}$ coefficient in the power series of $\sum _{k=1} ^j (\Theta (z) A^*) ^k =: \Phi (z)$, and observe that by
Proposition \ref{prop:Elliott}, $l_j (A) = \int _{\bm{T}} z^{-j} \Om _{\Theta _A} (dz)$, and that $l_j (\bm{1}) = V_n ^* P_- Z _\Theta ^{-j} P_- V_n$
since $( Z_\Theta ^{-j} b_i ^- , b_j ^- ) _\Theta = \int _{\bm{T}} \zeta ^{-j} ( \Om _\Theta (d\zeta ) e_i , e_j ) = ( l_j (\bm{1} ) e_i ,e_j )$
for all $j \in \bm{N} \cup \{ 0\}$.

\begin{lemming}
    The coefficients $l_k (A) $ obey the recurrence relations $l_k (A) = c_k A^* + \sum _{j=1} ^{k-1} c_{j} A^* l_{k-j} (A)
    = c_k A^* + \sum _{j=1} ^{k-1} l_{j} (A) c_{k-j}A^*$. \label{lemming:idone}
\end{lemming}

\begin{proof}
    By definition $c_k$ is the $k^{\mr{th}}$ coefficient of $\Theta (z)$ and $l_k (A) $ is the $k^{\mr{th}}$ coefficient
of $\Theta (z) A^* + ... + (\Theta (z) A^*) ^k$. Let $\Gamma _k $ denote the linear functional which picks out the $k^{\mr{th}}$
coefficient of a power series. Then clearly
\ba l_k (A) &  = & c_k A^* + \Gamma _k [ (\Theta (z) A^*) ^2 + ... + (\Theta (z) A^*) ^k ]  \nonumber \\
& = & c_k A^* + \Gamma _k [ (\Theta (z) A^*) \left( \Theta (z) A^* + ... + (\Theta (z) A^*) ^{k-1} \right) ]. \ea
    Let $b_j$ denote the coefficients in the power series of $\Phi (z) := \Theta (z) A^* + ... + (\Theta (z) A^* ) ^{k-1}$.
Since $\Theta (0) = 0 = c_0$, it is easy to see that $b_j = l_j (A)$ for all $1\leq j \leq k-1$. Hence it follows
that  $\Gamma _k [\Theta (z) A^* \Phi (z)] = c_1 A^* l_{k-1} (A) + c_2 A^* l_{k-2} (A)  + ... + c_{k-1} A^* l_1 (A) = \sum _{j=1} ^{k-1} c _{j} A^*
l_{k-j} (A)$. Also since $\Theta (z) A^*$ commutes with $\Phi (z)$ it follows that $\Gamma _k [\Theta (z) A^* \Phi (z) ]
= \Gamma _k [\Phi (z) \Theta (z) A^* ] = \sum _{j=1} ^{k-1} l _{k-j} (A) c_j A^* = \sum _{j=1} ^{k-1} l_j (A) c_{k-j} A^*$.
\end{proof}

The following combinatorial fact will be needed:

\begin{lemming}
    Let $(a_i )$, $(b_i)$, and $(c_i)$, $i=1, ... ,n$ be arbitrary sequences of (in general non-commuting) variables. Then the sum
$\sum _{i=1} ^{n-1} \sum _{j=1} ^{n-i-1} a_i b_j c_{n-i-j}$ is a rearrangement of the sum $\sum_{i=1} ^{n-1}
\sum _{j=1} ^{i-1} a_j b_{i-j} c_{n-i}$.  \label{lemming:combo}
\end{lemming}

    The above lemma can be established with a straightforward proof by induction. We omit the proof.
Let $l_j := l_j (\bm{1})$.

\begin{lemming}
    The $l_j$ and $l_j (A)$ obey the recurrence relation
\be l_k (A) = l _k A^* + \sum _{j=1} ^{k-1} l_j (\bm{1}) [ A^* - \bm{1} ] l_{k-j} (A) .\ee \label{lemming:formula}
\end{lemming}

\begin{proof}
    For convenience, let $q_n := l_n (A)$ for the remainder of the proof.
By Lemma \ref{lemming:idone}, we have that \be q_n = c_n A^* + \sum _{i=1} ^{n-1} c_i A^* q_{n-i}
\ \ \mr{and} \ \ c_n = l_n - \sum _{i=1} ^{n-1} l_i c_{n-i}. \label{eq:recur} \ee Substituting the second equation into the first yields
\be q_n  =  l_n A^* + \sum _{i=1} ^{n-1} l_i A^* q_{n-i}  - \left( \sum _{i=1} ^{n-1} l_i c_{n-i} A^*
+ \sum _{i=1} ^{n-1} \sum _{j=1} ^{i-1}  l_j c_{i-j} A^* q _{n-i} \right) .\ee To prove the lemma, we need to show
that
\be \sum _{i=1} ^{n-1} l_i q_{n-i} = \sum _{i=1} ^{n-1} l_i c_{n-i} A^*  + \sum _{i=1} ^{n-1} \sum _{j=1} ^{i-1}
l_j c_{i-j} A^* q_{n-i} . \ee Substituting equation (\ref{eq:recur}) into the left hand side of this expression
gives: \be  \sum _{i=1} ^{n-1} l_i \left( c_{n-i} A^* + \sum _{j=1} ^{n-i-1} c_j A^* q_{n-i-j} \right) = \sum _{i=1} ^{n-1} \left( l_i c_{n-i} A^* + \sum _{j=1} ^{i-1} l_j c_{i-j} A^* q_{n-i} \right) .\ee Canceling like terms
and applying the identity from Lemma \ref{lemming:combo} proves the claim.
\end{proof}

\begin{proof}{ (Proposition \ref{prop:acmeasure})}

    Proposition \ref{prop:Elliott} shows that $l_j (A) = \int _{\bm{T}} \ov{z} ^j \Om _{\Theta _A} (dz)$ for
all $j \in \bm{N} $. Hence to prove this proposition, it suffices to show that
$d_k := V_n ^* P_- Z(A) ^{-k} P _- V_n = \int _{\bm{T}} \ov{z} ^j V_n ^*  P_- \La _A (dz) P_-  V_n = l_k(A) $ for all $k \in \bm{N}$. The fact that
$d_k = l_k (A)$ for $k \in -\bm{N}$ will follow from taking adjoints, and since $\Theta (0) = 0$, it follows that
$\Om _\Theta (\bm{T}) = \bm{1}$ so that $d_0 = \bm{1} _n = l_0 (A)$. Thus if we can prove that $d_k = l_k (A)$ for all
$k \in \bm{N}$, then all moments of the measures $V_n ^* P_- \La _A P_- V_n $ and $\Om _{\Theta _A}$ agree so that they must be equal.

This will be accomplished by proving that the
$d_k$ obey the same recurrence formula as the $l_k (A)$ given in the previous lemma. For simplicity identify the standard basis
$\{ e_i \}$ of $\bm{C} ^n$ with the basis $\{ b_i ^- \} $ of $\mf{D} _- \subset L^2 _\Theta$, and let $P:= P_-$ so that we can
write $V_n ^* P_- \left( Z(A) ^* \right) ^k P_- V_n $ as $ P \left( Z(A) ^* \right) ^k P$.
The calculation proceeds as follows \small
\ba \footnotesize  P \left( Z(A) ^* \right) ^k P & = & P \left( Z ^{-1} + Z^{-1} P (A^*-\bm{1} ) P \right) \left( Z^{-1} + Z ^{-1} P (A^* - \bm{1} ) P \right) ^{k-1} P
\nonumber \\
& = & P Z^{-1} P (A^*- \bm{1} ) P \left( Z^{-1} + Z^{-1} P (A^*-1) P \right) ^{k-1} P + P Z^{-1} \left( Z^{-1} + Z^{-1} P (A^*-1) P \right) ^{k-1} P
\nonumber \\
& = & l_1 (A^*-\bm{1}) d_{k-1} + P Z^{-1} \left( Z^{-1} + Z^{-1} P (A^*-1) P \right) \left( Z^{-1} + Z^{-1} P (A^*-1) P \right) ^{k-2} P \nonumber \\
&= & l_1 (A^* - \bm{1}) d_{k-1} + P Z^{-2} P (A^* - \bm{1} ) P ( Z^{-1} + Z^{-1} P (A^*  - \bm{1} ) P )^{k-2} P
\nonumber \\
& & +P Z^{-2} ( Z^{-1} +Z^{-1} P (A^* - \bm{1} ) P) ^{k-2} P \nonumber \\
& = & l_1 (A^*-\bm{1}) d_{k-1} + l_2 (A^*-\bm{1}) d_{k-2}  + ... + l_{k-1} (A^*- \bm{1} ) d_1 \nonumber \\ &  & + P Z ^{-(k-1)} \left( Z^{-1} + Z^{-1} P (A^*-1) P \right) P
\nonumber \\
& = & l_k A^*  + \sum _{j=1} ^{k-1} l_j (A^* -\bm{1}) d_{k-j} .\ea \normalsize This is the same formula as in Lemma \ref{lemming:formula}. We conclude that
$d_k = l_k (A)$, and hence that $\Om _{\Theta _A}  =\La _A$.
\end{proof}

\subsection{The Weyl integral formula and proof of the disintegration theorem}

\label{subsection:Weylint}

    Let $T$ be a contraction on a separable Hilbert space $\mc{H}$. The defect operators $D _T$, $D_{T^*}$ are
defined by $D_T := \sqrt{ 1 - T^*T}$, the defect subspaces $\mf{D} _T , \mf{D} _{T^*}$ by $\mf{D} _T := \ov{ \ran {D_T} }$
and the defect indices by $\mf{d} _T := \dim{\mf{D} _T }$. We say a contraction $T$ has defect indices $(n,n)$
if $\mf{d} _T = n = \mf{d} _{T^*}$. Let $P_T$ denote the projection onto $\mf{D} _T$. Then $T_0 := T - T P_T$ is a partial
isometry with kernel $\mf{D} _T$ and with range the orthogonal complement of $\mf{D} _{T^*}$.

The Nagy-Foias characteristic function of a contraction $T$ is defined as
\be \Theta _T (z) = \left( -T + z D _{T^*} (\bm{1} - z T^* ) ^{-1} D_T \right) | _{\mf{D} _T},\ee and is a contractive analytic function
with domain $\mf{D} _T$ and range $\mf{D} _{T^*}$. Two contractions $T$, $T'$ are unitarily equivalent if and only if their characteristic
functions coincide, \emph{i.e.} if and only if there are isometries $U, V$ such that $U \Theta _T = \Theta _{T'} V$. It is straightforward
to check that $T$ is a partial isometry if and only if $\Theta _T (0) = 0$. It follows that if $T$ is a contraction with defect indices
$(n,n)$, then $\Theta _{T_0} (0) = 0$. In Section \ref{subsection:model} we will show that given any partial isometry $V$ with defect indices $(n,n)$, that
$\Theta _V$ coincides with $\Theta _{Z _{\Theta _V (0)}}$ so that $V$ is unitarily equivalent to $Z _{\Theta _V} (0)$. It will follow that
any contraction $T$ with defect indices $(n,n)$ is unitarily equivalent to some extension of the partial isometry $Z _{\Theta _{T_0}} (0)$.

Given $T$ and $T_0$, let $\mf{D} _+ := \mf{D} _T$ and $\mf{D} _- := \mf{D} _{T^*}$, and let $\{ \psi ^+ _i \} _{i=1} ^n$,
$\{ \psi ^- _i \} $ be orthonormal bases for $\mf{D} _\pm$. Fix an isometry $W$ of $\mf{D} _+$ onto $\mf{D} _-$ by $W \psi ^+ _i
= \psi ^- _i$. Now define for any $U \in U(n)$, $T(U) := T _0 + W \hat{U}$, where $\hat{U} : \mf{D} _+ \rightarrow \mf{D} _+$
is the bijective isometry defined by
\be \left( \ip{ \cdot}{\psi _1 ^+} , ... , \ip{\cdot}{\psi _n ^+} \right) \left[ U _{ij} \right] \left( \begin{array}{c} \psi _1 ^+ \\ \vdots
\\ \psi _n ^+ \end{array} \right) = \sum _{i,j =1} ^n U_{ij} \ip{\cdot}{\psi _i ^+} \psi _j ^+ .\ee

If $T=Z _\Theta$, this notation agrees with that of the previous section if we choose $\psi ^+ _i = b _i ^+$ and $\psi ^- _i = b_i ^-$.

Now any $U \in \mc{U} (n)$ can be written as $U = V^* D V$ where $V \in \mc{U} (n)$ and $D \in \bm{T} ^n$, \emph{i.e.} $D = \mathrm{diag} ( z_1, ..., z_n )$ with
$z_i \in \bm{T}$. Hence $U_{ij} = \sum _k z_k \ov{V_{ki}} V_{kj}$ and we can write
\be \hat{U} = \sum _{ijk =1} ^n z_k \ov{V _{ki}} V_{kj} \ip{\cdot}{\psi ^+ _i} \psi _j ^+
=: z_1 R_1 + z_2 R_2 + ... + z_n R_n , \ee and \be T(U) = R_0 + z_1 R_1 +... z_n R_n.\ee Here $R_0 := T(0) = T_0$ and
for $i\geq 1$ the $R_i$ are all finite rank operators depending on $V$ and not on $D$, \emph{i.e.} the $R_n = R_n (V)$ are independent of the $z_i \in \bm{T}$.

Hence for any polynomial $p (z) = \sum _{k=0} ^j p_k z^k$, it follows that
\be p (T(U)) = p (T (V^* D V)) = \sum _{i_1 , ... , i_n =0} ^k z_1 ^{i_1} ... z_n ^{i_n} A_{i_1,..., i_n} (V) \label{eq:brkdwn},\ee
where the coefficient operators $A_{i_1,..., i_n} (V)$ depend only on $V$, and so are constant if $V$ is fixed.

Weyl's integration formula for $\mc{U} (n)$ (see \emph{e.g} \cite{Rossman}) states that if $H$ is Haar measure on $\mc{U} (n)$, $\bm{T} ^n$ denotes the subgroup of diagonal unitary
matrices, $G := \mc{U} (n) / \bm{T} ^n$, and $H_G$ Haar measure on $G$ then:

\begin{thm}{ (Weyl Integration Formula) }
    If $f$ is a continuous function on $\mc{U} (n)$, then,
    \be \int _{\mc{U} (n)} f(U) dH(U) = \frac{1}{n!} \int _G \left( \int _{\bm{T} ^n } f(VDV^*) \Delta (D) \ov{\Delta (D)} dD \right) d H_G (V \bm{T} ^n) .\ee
\end{thm}

    In the above if $D = \mr{diag} (z_1, ... z_n)$ then $dD := dz_1 ... dz_n$, and $\Delta (D) := \prod _{j<k} (z_j - z_k)$.
The following fact is a straightforward consequence of Weyl's integration formula

\begin{prop}
    If $T$ is a completely non-unitary contraction with defect indices $(n,n)$, and $f = \ov{h} +g$ for $h,g \in H ^\infty (\bm{T})$,
then
\be \int _{\mc{U} (n)} f(T(U)) dH(U) = f(T(0)). \ee \label{prop:yayweyl}
\end{prop}

    Here if $h \in H^\infty $, and $T$ is a completely non-unitary contraction, then $\ov{h} (T)$ is defined as $h ^* (T^*)$, where
$h^* (z) = \ov{h(\ov{z})} \in H^\infty$.

\begin{proof}
    It suffices to establish the formula in the case where $f=p$ is a polynomial. The more general formula follows by taking
adjoints, and limits with the aid of the $H^\infty$ functional calculus for completely non unitary contractions (see \emph{e.g.}
\cite{Foias}). For fixed
$V \in \mc{U} (n)$, equation (\ref{eq:brkdwn}) implies that
\be p (T(U)) = p (T(0)) + \sum  ' z_1 ^{i_1} ... z_n ^{i_n} A_{i_1,..., i_n} (V), \ee where the prime denotes that
the sum is taken over all values of the $i_1, ... i_n$ where at least one of the $i_j ; 1 \leq j \leq n $ is non-zero.

By Weyl's integration formula, \be \int _{\mc{U} (n)} p(T(U)) dH(U)
= p(T(0)) + \sum ' A_{i_1, ... i_n} \int _G \left( \int _{\bm{T}} ... \int _{\bm{T}} z_1 ^{i_1}...z_n ^{i_n} \left|  \prod _{j<k} (z_j -z_k) \right| ^2
dz_1 ... dz_n \right) dH_G (V\bm{T} ^n ) .\ee Hence, to prove the proposition, it suffices to show that provided at least one of the
$i_1, ... , i_n$ is non-zero, that
\be 0= \int _{\bm{T}} ... \int _{\bm{T}} z_1 ^{i_1}...z_n ^{i_n} \left|  \prod _{j<k} (z_j -z_k) \right| ^2
dz_1 ... dz_n . \label{eq:vanish} \ee This is straightforward to show. If one expands out  $\left|  \prod _{j<k} (z_j -z_k) \right| ^2$ one
obtains a sum of terms of the form $z_1 ^{j_1} ... z_n ^{j_n}$ where $j_1 + j_2 +... + j_n =0$. It follows that the product
$ z_1 ^{i_1}...z_n ^{i_n} \left|  \prod _{j<k} (z_j -z_k) \right| ^2 $is the sum of terms of the form $z_1 ^{k_1} ... z_n ^{k_n}$
where $k_1 + ... + k_n \geq 1$, and therefore there is at least one $k_i \geq 1$ in each such term.
But then it is clear that the above integral in (\ref{eq:vanish}) vanishes, since if $k_l \geq 1$, then
\be \int _{\bm{T}} ... \int _{\bm{T}} z_1 ^{k_1}...z_l ^{k_l}...z_n ^{k_n} dz_1 ... dz_n  = 0 .\ee
\end{proof}

There is one final observation to make before presenting the matrix-valued disintegration theorem. If $\Om _{\Theta _A}$ are the
Aleksandrov-Clark measures discussed in the previous section, recall that,

\be B_A (z) = \frac{\bm{1} -\Theta (z) A^*}{\bm{1} - \Theta (z) A^*} = \int _\bm{T} \frac{\zeta +z}{\zeta -z} \Om _{\Theta _A} (d\zeta ).\ee

Taking $A=0$ shows that $\bm{1} _{\bm{M} _n } = \int _\bm{T} \frac{\zeta +z}{\zeta -z} d\Om _0 (\zeta ).$ Letting
$m$ denote the diagonal positive matrix valued measure given by $n$ copies of Lebesgue measure on the diagonal, then
\be \int _{\bm{T}} \frac{\zeta +z}{\zeta -z} m (d\zeta) = \sum _{k=0} ^\infty \int _{\bm{T}} \left( (\ov{\zeta} z) ^k
+ (\ov{\zeta} z)  ^{k+1}  \right) m (d\zeta) = \bm{1} _n .\ee By the uniqueness of the representing measure in the matrix-valued
Herglotz theorem \cite[Theorem 3]{Elliott}, it follows that $\Om _0 = m$.

\begin{thm}
    Let $\Theta$ be a $\bm{M} _n$-valued contractive analytic function on $\bm{D}$, and $\Om _{\Theta _A}$ the AC measures
associated with $\Theta$ for any $A \in \ov{(\bm{M} _n) _1}$. Then $\Om _0 = m$ and for any continuous function $f$ on $\bm{T}$,

\be \int _{\mc{U} (n)} \int _{\bm{T}} f(\zeta ) \Om _{\Theta _U} (d\zeta ) dH (U) = \int _\bm{T} f(\zeta) m (d\zeta). \ee

\label{thm:disintegrate}

\end{thm}


\begin{proof}
    Recall that if $A \in \ov{ (\bm{M} _n ) _1} $, $\La _A$ denotes the positive operator valued measure obtained as the
compression of the projection valued measure of the unitary dilation of $Z(A)$ to $L^2 _\Theta (\bm{T})$. In the case
where $A=U$ is unitary, $\La _U$ is the projection-valued measure obtained from $Z _\Theta (U)$.

Now by the previous proposition, Proposition \ref{prop:yayweyl}, if $f = \ov{q} + p$ where $p, q$ are polynomials, then
\be \int _{\mc{U} (n)} \int _{\bm{T}} f(\zeta )  \La _U (d\zeta ) dH (U) =  \int _{\mc{U} (n)} f(Z(U)) dH(U) = q^* (Z(0) ^*)  + p(Z(0)).\ee

By Proposition \ref{prop:acmeasure}, if we again identify $\bm{C} ^n$ with $\mf{D} _- \subset L^2 _\Theta$ and let $P := P_-$,
the projector of $L^2 _\Theta$ onto $\mf{D} _-$, then $P  \La_A P  = \Om _{\Theta _A}$, so that

\ba \int _{\mc{U} (n)} \int _{\bm{T}} f(\zeta ) \Om _U (d\zeta ) dH (U) &  = &  \int _{\mc{U} (n)}  \int _{\bm{T}} f(\zeta ) P \La _U  (d\zeta ) P dH (U) \nonumber
\\ & = & P  \int _{\mc{U} (n)} \int _{\bm{T}} f(\zeta ) \La _U (d\zeta ) dH (U) P \nonumber \\
& = & P \int _{\mc{U} (n)} f ( Z_\Theta (U) ) dH (U) P \nonumber \\
& = &  P ( q^* (Z(0)^*)  + p(Z(0)) ) P
\nonumber \\ & = &  \int _\bm{T} f(\zeta ) \Om _0 (d\zeta) \nonumber \\
& = &   \int _\bm{T} f(\zeta) m (d\zeta). \ea For $f$ an arbitrary continuous function, the statement follows by approximating
$f$ by functions $f_n$ of the form $f_n = \ov{q_n} + p_n$ since such functions are dense in the Banach space of continuous functions
on $\bm{T}$.
\end{proof}

\section{The Cauchy integral representation of $K^2 _\Theta$}

\label{section:cauchyint}

    Let $H^2 _\Theta$ be the closure of the polynomials in $L^2 _\Theta$, \emph{i.e.} the closed subspace of $L^2 _\Theta$ generated
by $Z_\Theta $ and $\mf{D} _-$, the constant functions. In this section we construct an isometry $V_\Theta : H^2 _\Theta \rightarrow K^2 _\Theta$,
where $K^2 _\Theta $ is the de Branges-Rovnyak space associated with $\Theta$ and show that the image of $Z_\Theta ^*$ under this
transformation is a rank-$n$ perturbation of $X_\Theta$, the restriction of the backwards shift from $H^2 _n (\bm{D} )$ to $K^2 _\Theta$. In this section
we do not assume that $\Theta (0) = 0$ in general.

\subsection{de Branges-Rovnyak spaces}

    Let $L^2 _n (\bm{T})$ denote the Hilbert space of $\bm{C} ^n$-valued functions which are square integrable
with respect to normalized matrix-valued Lebesgue measure $m$ on $\bm{T}$, and recall that $H^2 _n (\bm{D} ) \subset
L^2 _n (\bm{T})$ is the subspace of $\bm{C} ^n$ valued functions which are analytic in $\bm{D}$ and whose $L^2$ norm on circles of
radii $r<1$ remains bounded as $r \rightarrow 1$.

    Given $\Theta \in \left( H^\infty _{\bm{M} _n } (\bm{D} ) \right) _1$, the de Branges-Rovnyak space $K^2 _\Theta$ is
defined as follows. Let $P_{H^2}$ denote the projection of $L^2 _n (\bm{T})$ onto $H^2 _n (\bm{D})$, and let $T _\Theta$ denote
the operator of multiplication by $\Theta $ on $H^2 _n (\bm{D})$, $T_\Theta f = \Theta f$ for all $f \in H^2 _n (\bm{D})$. The
de Branges-Rovnyak space $K^2 _\Theta $ is defined as the range of $R_\Theta := \sqrt{\bm{1} - T_\Theta T_\Theta ^* } $
endowed with the inner product that makes $R_\Theta$ a co-isometry of $H^2 _n (\bm{D})$ onto its range.  Hence if $f,g \in H^2 _n (\bm{D})$ and at
least one of $f,g$ is orthogonal to the kernel of $R_\Theta$, then $\ip{R_\Theta f}{R_\Theta g} _\Theta = \ip{f}{g}$, see \cite{Sarason-dB} for more details.
We will denote the inner product in $K^2 _\Theta$ by $\ip{\cdot}{\cdot} _\Theta$ to distinguish it from the inner product of $H^2 _\Theta$
which is denoted by $( \cdot , \cdot ) _\Theta $. For $z, w \in \bm{D}$, let
\be \Delta _w (z) := \frac{\bm{1} - \Theta (z) \Theta (w) ^*} {1 - z \ov{w}} \label{eq:opkern} ,\ee be the matrix kernel function at $w$. The Hilbert space
$K^2 _\Theta$ is the closed linear span of the point evaluation functions \be \delta_z ^{\vec{x}} := \Delta _z \vec{x}, \ee for $\vec{x} \in \bm{C} ^n$ and
$z \in \bm{D}$. The notation $\delta_z ^j := \delta_z ^{e_j}$ where $\{ e_j \}$ as before is an ON basis of $\bm{C}^n$ will sometimes be used. Inner products with $\delta_z ^{\vec{x}}$ gives point
evaluations at $z \in \bm{D}$:
\be \ip{f}{\delta_z ^{\vec{x}}} _\Theta = ( f (z) , \vec{x} ) _{\bm{C} ^n}, \ee for any $f \in K^2 _\Theta$.

    We will now discuss the Cauchy integral representation for vector-valued de Branges-Rovynak spaces $K^2 _\Theta$. This
will be a straightforward generalization of the methods of \cite[Chapter III]{Sarason-dB}. Since most of the arguments
generalize with only trivial modifications, many of the results will be stated without proof.

\subsection{The Cauchy integral representation of $K^2 _\Theta$}

Recall that $\Om _\Theta $ is the unique positive $\bm{M} _n$-valued measure on $\bm{T}$ associated with the purely contractive
$\Theta$ by the Herglotz theorem.

One defines the Cauchy integral of $\Om _\Theta$ by \be C \Om _\Theta (z) := \int _\bm{T} \frac{1}{1-\ov{\zeta} z} \Om _\Theta (d\zeta) . \ee
This is clearly a $\bm{M} _n$-valued function which is analytic in $\bm{D}$.
Next for any $f \in L^2 _\Theta (\bm{T})$ define the Cauchy integral of $f$ by
\be C _\Theta f (z) := \int _\bm{T} \frac{1}{1-\ov{\zeta} z} \Om _\Theta (d\zeta) f (\zeta) .\ee For each such $f$ this is
an analytic $\bm{C} ^n$-valued function on $\bm{D}$.

By definition $C_\Theta f (z) = \left( f , k _z \right) _\Theta$ where
\be k_z (\zeta) := (1 -\ov{z} \zeta) ^{-1}, \ee $k_z \in L^2 _\Theta$ for $z \in \bm{D}$.  Hence the kernel of the map $C_\Theta $ is the orthogonal complement of the span of the kernel functions $k_z$, $z \in \bm{D}$ in $L^2 _\Theta$. The closed linear span of the kernel functions $k_z ^{\vec{x}}
:= k_z \vec{x} $ where $\vec{x} \in \bm{C} ^n$
is easily seen to be the span of the polynomials in $L^2 _\Theta$ which we defined previously to be $H^2 _\Theta$.

The following Lemma is easy to verify and its proof is omitted:

\begin{lemming}
The following identity holds:
\be \int _{\bm{T} } \frac{1}{1-\ov{a} z} \frac{1}{1-b\ov{z}} \Om _\Theta (dz) = (\bm{1} - \Theta (b) ) ^{-1}
\Delta _a (b) (\bm{1} - \Theta (a) ^* ) ^{-1}. \ee \label{lemming:id2}
\end{lemming}

Given $f \in H^2 _\Theta$, define $V_\Theta f (z) := (1 - \Theta (z) ) C _\Theta f(z) $. We will write $k_z ^i$ for $k_z ^{e_i}$ where
$\{ e_i \}$ is the canonical ON basis for $\bm{C} ^n$.  Then observe that
by applying the above lemma,
\ba V_\Theta k_a ^i (z) & = & (1 -\Theta (z) ) \int _{\bm{T}} \frac{1}{1-\ov{a} w} \frac{1}{1-z \ov{w}} \Om _\Theta (dw) e_i
\nonumber \\ &  & \Delta _a (z) ( \bm{1} -\Theta (a) ^* ) ^{-1} e_i . \label{eq:Vconst} \ea This shows that $V_\Theta k_a ^i$ is a linear combination
of the point evaluation functions $\{ \delta_a ^j\} _{j=1} ^n \subset K^2 _\Theta$ so that $V_\Theta$ is a linear map from $H^2 _\Theta$ into
$K^2 _\Theta$. Here, as above $\delta _a ^j = \delta _a ^{e_j} = \Delta _a e_j$.

\begin{prop}
The linear map $V_\Theta : H^2 _\Theta \rightarrow K^2 _\Theta$ is an isometry of $H^2 _\Theta$ onto $K^2 _\Theta$.
\end{prop}

\begin{proof}
    For any $a \in \bm{D}$, $(\bm{1} - \Theta (a) ^* )$ is invertible so that $\{ (\bm{1} -\Theta (a)  ^* ) e_i \} _{i=1} ^n$ is a basis
for $\bm{C} ^n$. It follows that the span of the set of functions $S := \{ \hat{k} _a ^j  \ | \ a \in \bm{D} , \ 1\leq j\leq n \} \subset H^2 _\Theta $ where
$\hat{k} _a ^j (z) = \frac{1}{1-\ov{a}z} (\bm{1} -\Theta (a) ^* ) e_j $ is equal to the span of the $\{ k _a ^j \ | \ a \in \bm{D} ,
\ 1 \leq j \leq n \} $. The span of the last set is dense in $H^2 _\Theta$, and hence so is the span of $S$. Hence to prove
that $V_\Theta $ can be uniquely extended to an isometry of $H^2 _\Theta$ onto $K^2 _\Theta$, it suffices to show that
\be \left( \hat{k} _a ^i , \hat{k} _b ^j \right) _\Theta = \ip{ V _\Theta \hat{k} _a ^i }{ V_\Theta \hat{k} _b ^j } _\Theta . \label{eq:ipeq} \ee
The left hand side of the above equation is equal to
\be \left( \int _\bm{T} \frac{1}{1-\ov{a}z} \frac{1}{1-b\ov{z}} \Om _\Theta (dz) (\bm{1} -\Theta (a) ^* ) e_i ,
(\bm{1} -\Theta (b) ^* ) e_j \right) = ( \Delta _a (b) e_i , e_j ) , \ee by the previous lemma.

By the calculation preceding this proposition, $V_\Theta \hat{k} _a ^i (z) = \delta _a ^i (z) $ so that the
right hand side of equation (\ref{eq:ipeq}) is equal to $ \ip{ \delta _a ^i}{\delta _b ^j} _\Theta $. Since these are the point evaluation
functions in $K^2 _\Theta $, this is equal to $( \delta _a ^i (b)  , e_j ) = ( \Delta _a (b) e_i , e_j)$. Hence
both sides are equal and $V_\Theta $ is an isometry.

\end{proof}

Let $Y_\Theta (A) := Z _\Theta  (A) | _{H^2 _\Theta}$ (it is clear that $H^2 _\Theta$ is invariant for each $Z_\Theta (A)$),
$A \in (\bm{M} _n ) _1$, and let $X _\Theta := S^* | _{K^2 _\Theta }$ where $S$ is the shift (multiplication by $z$) in $H^2 _n (\bm{D})$.

\begin{prop}
The compressed backwards shift $X_\Theta$ and $Y_\Theta ^*$ are related as follows:
\be X_\Theta V_\Theta = V_\Theta Y_\Theta ^* \left( \bm{1} - ( \bm{1}_n -\Theta (0) ) \sum _{i=1} ^n \left( \cdot , b_i ^- \right) _\Theta b_i ^- \right) .\ee
\label{prop:twine}
\end{prop}

    In the above statement, $Y_\Theta := Y _\Theta (\bm{1} )$. If we let $P = P_- = \sum _{i=1} ^n \left( \cdot , b_i ^- \right) _\Theta b_i ^-$, the projector onto the constant functions in $H^2 _\Theta$, then the claim can be written \be V_\Theta ^* X _\Theta V_\Theta =
    Y_\Theta ^* \left( 1 -P(1 -\Theta (0)) P \right) \label{eq:opex}.\ee

This proof of this proposition is an obvious $n-$dimensional generalization of calculations in \cite{Sarason-dB}.

\begin{proof}
For simplicity identify the fixed basis of $\bm{C} ^n$, $\{ e_i \}$ with the basis $\{ b_i ^- \}$ for $\mf{D} _-$, the constant functions in $L^2 _\Theta$.  This basis is orthonormal if $\Theta (0) =0$. Given any $f \in H^2 _\Theta$, consider $ V_\Theta Y _\Theta ^* f (z) = (\bm{1} -\Theta (z)) C_\Theta Y _\Theta ^* f (z)$.
First as in \cite{Sarason-dB} it is easy to calculate that
\ba ( C_\Theta Y _\Theta ^* f (z) , e_i ) _{\bm{C} ^n} & = & \int _\bm{T} \frac{1}{1-\ov{w}z} (\Om _\Theta (dw) Y^* _\Theta f(w), e_i ) \nonumber \\
& =& \left( Y_\Theta ^* f , k_z ^i \right) _\Theta = \left( f , Z_\Theta k_z ^i \right) _\Theta \nonumber \\ & = & \int _\bm{T} \frac{\ov{w}}{1-\ov{w}z}
 (\Om _\Theta (dw) f(w) , e_i ) = \frac{1}{z} ( C_\Theta f(z) - C_\Theta f(0) ) .\ea

It follows that
\ba V_\Theta Y_\Theta ^* f(z) & = &  (\bm{1} - \Theta (z) ) \frac{ C_\Theta f (z) -C _\Theta f (0) }{z} \nonumber \\
& = & S^* V_\Theta f (z) + (S^* \Theta (z) ) C _\Theta f (0) \nonumber ,\ea and hence that
\be V_\Theta Y_\Theta ^* f = S^* V_\Theta f + S^* \Theta (C_\Theta f (0) ) .\ee

Applying this formula to the case where $f = e_i$, and using equation (\ref{eq:Vconst})  yields
\be V_\Theta Y _\Theta ^* e_i = S^* \Delta _0 ( \bm{1} - \Theta (0) ^* ) ^{-1} e_i - S^* \Theta \int _\bm{T} \Om _\Theta (dw) e_i.
\label{eq:actonconst} \ee

Short calculations show that $S^* \Delta _0 = -(S^* \Theta) \Theta (0)^*$ while Lemma \ref{lemming:id2} implies that
$\int _{\bm{T} } \Om _\Theta (dw) = (\bm{1} -\Theta (0) ) ^{-1} (\bm{1} -\Theta (0) \Theta (0) ^* ) ( \bm{1} -\Theta (0) ^* )^{-1}$.
Substituting these formulas into equation (\ref{eq:actonconst}) and simplifying leads to

\be V_\Theta Y _\Theta ^* e_i = S^* \Theta (\bm{1} -\Theta (0) )^{-1} e_i ,\ee or equivalently that
\be S^* \Theta e_i = V_\Theta Y _\Theta ^* (\bm{1} -\Theta (0) ) e_i .\label{eq:id3} \ee

Since \be C_\Theta f (0) = \sum _{i=1} ^n ( C_\Theta f (0) , e_i) e_i  = \sum _{i=1} ^n \int _\bm{T}
(\Om _\Theta (dw) f , e_i ) e_i = \sum _{i=1} ^n \left( f , e_i \right)_\Theta e_i ,\ee it follows that
\ba V_\Theta Y _\Theta ^* f & = & S^* V_\Theta f + S^*\Theta (C _\Theta f ) (0) \nonumber \\
& = & S^* V_\Theta f + \sum _{i=1} ^n \left( f , e_i \right) _\Theta V_\Theta Y_\Theta ^* (\bm{1} - \Theta (0) ) e_i \nonumber \\
&= & X_\Theta V_\Theta ( \bm{1} -\sum _{i=1} ^n \left( \cdot , e_i \right)_\Theta (\bm{1} -\Theta (0) ) e_i ) f .\ea
\end{proof}

\subsection{Extreme points}

    In the case where $\Theta$ is scalar-valued, it is well known that $\Theta$ is an extreme point of the unit ball of $H^\infty$
if and only if $1 - |\Theta|$ fails to be log-integrable \cite[pgs. 138-139]{Hoffman}, and that this happens if and only if $H^2 _\Theta = L^2 _\Theta$. These
facts follow easily from Szeg$\mr{\ddot{o}}$'s theorem \cite[pgs. 49-50]{Hoffman} and the fact that the derivative of the absolutely continuous
part of $\Om _\Theta$ with respect to Lebesgue measure is $\frac{1- |\Theta | ^2 }{|1-\Theta |^2}$.

These facts generalize almost verbatim to the case where $\Theta$ is $\bm{M} _n$-valued and purely contractive. First,
it is easy to check \cite[Theorem 9]{Elliott} that the derivative of the absolutely continuous part of $\Om _\Theta$
with respect to Lebesgue measure is \be W_\Theta  (\zeta ) = \left( \bm{1} - \Theta (\zeta )  \right) ^{-1} \left( \bm{1} -\Theta (\zeta) \Theta (\zeta ) ^* \right) \left(  \bm{1} - \Theta (\zeta )^*  \right) ^{-1}  .\ee

By the Helson-Lowdenslager generalization of Szeg$\mr{\ddot{o}}$'s Theorem,
\be \exp \left( \int _\bm{T} \mr{tr} ( \ln (W_\Theta (\zeta) ) m (d\zeta ) ) \right) = \inf _{A_0 , P} \int _{\bm{T}} \mr{tr}
\left( (A_0 + P(\zeta ) ) ^* (A_0 +P(\zeta ) )  \Om _\Theta (d\zeta ) \right) \label{eq:Helsons} .\ee Here the infimum is taken over all $n \times n$ matrices
$A_0$ of determinant one, and all polynomial matrix functions $P (z) = \sum _{j=1} ^k A_j z^j$, $A_j \in \bm{M} _n$
for $z \in \bm{D}$ which vanish at the origin \cite[Theorem 8]{Helson}.

With this fact in hand, and the fact that if $A,B$ are positive definite matrices the identity $\tr (\ln AB ) = \tr (\ln A) + \tr (\ln B)$ holds (this follows from the multiplicative property of the determinant), one can show as in the scalar case that $H^2 _\Theta = L^2 _\Theta$ if and only if $\int _\bm{T} \mr{tr} \left( (\bm{1} - | \Theta (z) | ) m(dz) \right) = - \infty$.  Indeed, in this case the left hand side of equation (\ref{eq:Helsons}) vanishes, and this implies that if $\mc{A} _0 ^* := \bigvee _{k \in \bm{N}} Z_\Theta ^{-k} \mf{D} _-$ that $\mf{D} _- \subset \ov{\mc{A} _0 ^*}$ and hence that $\mf{D} _- \subset \ov{\mc{A} _0}$ where $\mc{A} _0 := \bigvee _{k\in \bm{N}} Z_\Theta ^k \mf{D} _-$.  This readily leads to the conclusion that $H^2 _\Theta = L^2 _\Theta$. Moreover, using the fact that by \cite[Theorem 9]{Helson}, $\int _{\bm{T}} \tr \left( \ln | \Theta (z) | ^2 m (dz) \right) > - \infty$, it is easy to generalize the proof characterizing extreme points of the unit ball of $H^\infty (\bm{D})$ \cite[pgs. 138-139]{Hoffman}
to obtain an analogous characterization of extreme points of the unit ball of $H^\infty _{\bm{M} _n} (\bm{D})$. In summary one
can establish the following without difficulty:

\begin{thm}
    Given $\Theta \in \left( H^\infty _{\bm{M} _n} (\bm{D}) \right) _1$, the following are equivalent: \\
    (i) $\Theta $ is an extreme point. \\
    (ii) $\int _{\bm{T}} \tr \left( \ln ( \bm{1} - | \Theta (z) |) m (dz) \right) = -\infty$ \\
    (iii) $L^2 _\Theta = H^2 _\Theta$
\end{thm}

\subsubsection{Remark} For brevity we will say that $\Theta $ is extreme if it is an extreme point of the unit ball
of $H^\infty _{\bm{M} _n } (\bm{D} )$. In this case since $H^2 _\Theta = L^2 _\Theta$ we have that $Y_\Theta = Z_\Theta | _{H^2 _\Theta}
= Z_\Theta$ in Proposition \ref{prop:twine}.

\subsection{Determination of AC measures}

    The Cauchy integral representation of $K^2 _\Theta$ provides an another way of proving that $\Om _{\Theta _U} =
\La_U$ for $U$ unitary that is independent of the methods used in Section \ref{subsection:ACmeas}. In this subsection we do this and prove that $Z_\Theta (U)$ is unitarily equivalent to $Z_{\Theta _U}$.
Recall that $\Theta _U = \Theta U^*$.

Suppose that $\Theta (0) = 0$. Consider the subspace $K_0$ of $K^2 _\Theta$ spanned by the point evaluation functions
at $z=0$, $\delta ^j _0$, $1 \leq j \leq n$ where $\delta _0 ^j (z) = \Delta _0 (z) e_j = (1 - \Theta (z) \Theta (0) ^*) e_j
= e_j$ since $\Theta (0) = 0$. Then from earlier calculations we see that if $P =P_-$ denotes the
projection onto the constant functions in $H^2 _\Theta$ and $Q$ the projector onto $K_0$, the constant functions in $K^2 _\Theta$
then $V_\Theta P = Q V_\Theta$.
Let $R_\Theta$ denote the projection of $L^2 _\Theta$ onto $H^2 _\Theta$. Before
we defined $Z_\Theta (A) = Z_\Theta + P (A - \bm{1} ) P Z_\Theta $. Since $H^2 _\Theta$ is invariant for $Z_\Theta (A)$
it follows that
\ba Y_\Theta (A) ^* & = & R_\Theta Z_\Theta (A) ^* R_\Theta = R_\Theta Z_\Theta ^* R_\Theta + R_\Theta Z_\Theta ^* P (A ^* -\bm{1}) P R_\Theta
\nonumber \\ & = & Y_\Theta ^* + Y_\Theta ^* P (A^* -\bm{1}) P .\ea

Now by the intertwining relation of Proposition \ref{prop:twine}, $V_\Theta Y_\Theta V_\Theta ^* = X_\Theta
+ V_\Theta Y _\Theta ^* P V_\Theta ^* $. As calculated previously in equation \ref{eq:id3},
$V_\Theta Y _\Theta ^* e_i = S^* \Theta e_i$.
Hence we get that $V_\Theta Y_\Theta ^* P V_\Theta ^* = S^* \Theta V_\Theta P V_\Theta ^* = S^* \Theta  Q $.
This shows that \be V_\Theta Y_\Theta ^* V_\Theta ^* = X_\Theta + S^* \Theta Q \label{eq:shiftimage}, \ee and hence that
\be V_\Theta Y_\Theta (A) ^* V_\Theta ^* = (X_\Theta + S^* \Theta Q ) \left( \bm{1} + Q (A^* - \bm{1} )Q
\right) = X_\Theta +S^* \Theta QA^*Q \label{eq:image} .\ee Note that here the operator $QAQ$ denotes the operator
$\sum _{ij =1} ^n \ip{\cdot}{\delta _0 ^i}_\Theta  A_{ij} \delta _0 ^j$ where $\{ \delta _0 ^j =e_j \}$ is an ON
basis for the constant functions $K_0 \subset K^2 _\Theta$ and $A \in \ov{ (bm{M} _n ) _1 }$. In particular we conclude that
$Y _\Theta (0) ^*$ is unitarily equivalent to $X _\Theta$ (under our assumption that $\Theta (0) =0$). If $\Theta $ is extreme then also $Y _\Theta (A) = Z_\Theta (A)$.

    Now recall that the de Branges-Rovnyak spaces $K^2 _\Theta $ are the ranges of $R_\Theta = \sqrt{ \bm{1} - T_\Theta
T_\Theta ^*}$. If we define $\Theta _A := \Theta A^*$ for
$A \in \ov{ (\bm{M} _n ) _1 }$, then it follows that $K^2 _{\Theta _U} = K^2 _\Theta$ for unitary $U$.

\begin{lemming}
    Given any $U \in \mc{U} (n)$, let $W_U := V_\Theta ^*  V_{\Theta _U}$. Then $W_U Y _{\Theta _U} ^* = Y _\Theta (U) ^* W_U $.
    \label{lemming:intertwiner}
\end{lemming}

\begin{proof}
    By previous calculations,
\be V_{\Theta _U} Y_{\Theta  _U } ^* V_{\Theta _U} ^* = X + S^*\Theta _U Q  = X + S^* \Theta U^* Q.\ee Here $X := X_\Theta = X_{\Theta _U}$ acts
on $K^2 _\Theta = K^2 _{\Theta _U}$. But by equation (\ref{eq:image}) this agrees with $V_\Theta Y_\Theta (U) ^* V_\Theta ^*$.
\end{proof}

\begin{prop}
Suppose that $\Theta (0) = 0$. For any $U \in \mc{U} (n)$, let $\Om _{\Theta _U}$ be the measure associated with $\Theta _U := \Theta U^*$ by the
Herglotz theorem, and let $\La _U$ denote the $\bm{M} _n$ valued positive measure on $\bm{T} $ defined by $\La _U (I) := [ \left( \chi _I ( Z _\Theta (U) ) e_i,
e_j \right) _\Theta ]$. Then $\Om _{\Theta _U} = \La _U$. \label{prop:weakprop}
\end{prop}

\begin{proof}
    Clearly the claim holds for $U = \bm{1}$. Now suppose $U \neq \bm{1}$. Recall that $H^2 _\Theta$ is invariant for $Z_\Theta (U)$ and that
$Y _\Theta (U) := Z_\Theta (U) | _{H^2 _\Theta}$. By the previous lemma, there is a unitary operator $W_U$ which intertwines $Y_U ^* := Y_{\Theta _U} ^*$
and $Y (U) ^* := Y _\Theta (U) ^*$. Since $H^2 _\Theta$ is invariant for $Y (U)$, it is semi-invariant for $Y^* (U)$, for any $U \in \mc{U} (n)$.
Recall here that a subspace $S$ of a Hilbert space $\mc{H}$ is said to be semi-invariant for a semigroup of operators $\mf{S}$ if $S = S_1 \ominus S_2$
where $S_1 \supset S_2$ are invariant subspaces for $\mf{S}$. If $S$ is semi-invariant for the semigroup $\mf{S}$, then the compression of $\mf{S}$
to $S$ is a semigroup of operators on $S$ \cite{Sarason}.

Moreover it is not hard to show that $Z^* _\Theta (U)$ is the minimal unitary dilation of $Y^* (U)$. To prove this, it suffices to show that the linear
span of $Z (U) ^{-k} H^2 _\Theta $, for $k \in \bm{Z}$ is dense in $L^2 _\Theta$. Recall that $P$ projects onto the constant functions in $H^2 _\Theta$.
Now $\ran {Z (U) ^{-1} P = Z^{-1} U^* P} \supset \ran{Z^{-1} P}$, and $Z (U) ^{-2} P = Z^{-2} P + Z^{-1} P (U^* -\bm{1} ) P Z^{-1} P $. Since the range
of the second
term is contained in $\ran{Z^{-1} P} \subset \ran{Z^{-1} (U) P}$, it follows that the range of $Z ^{-2} P$ is contained in the closed linear span of the
ranges of $Z^{-2} (U) P$ and $Z^{-1} (U) P$. Continuing
in this fashion we get that $\bigvee _{k \in \bm{Z}} \ran{Z^k P} \subset \bigvee _{k \in \bm{Z}} \ran{Z ^k (U)}$. Since the first set is dense in $L^2 _\Theta$, so is the second so that $Z^* (U)$ acting on $L^2 _\Theta$ is indeed the minimal unitary dilation of $Y^* (U)$ acting on $H^2 _\Theta$. The same argument shows
that $Z^* _U$ is the minimal unitary dilation of $Y^* _U$.

Since there is a unitary $W_U$ intertwining $Y^* (U)$ and $Y^* _U$, the intertwiner version of the commutant lifting theorem \cite[pg. 66]{Paulsen} implies that there
is a unitary $\hat{W} _U : L^2 _{\Theta _U} \rightarrow L^2 _\Theta$ such that $\hat{W} _U | _{H^2 _{\Theta _U}} = W_U$ and such that
$\hat{W} _U Z ^* _U = Z^* (U) \hat{W}$. If $P_U$ denotes the projector onto the constant functions in $H^2 _{\Theta _U}$,
then, by construction $W_U P_U = P W_U$ since $W_U = V^* _\Theta V_{\Theta _U}$, and it is clear that $\hat{W} _U$ obeys the
same formula, $\hat{W} _U P_U = P \hat{W} _U$. In particular if $\{ b_i ^- \}$ is the canonical ON basis of $\mf{D} _-$ in $H^2 _\Theta$
and $\{ \beta _i ^- \}$ is the corresponding basis in $H^2 _{\Theta _U}$, then $\hat{W} _U \beta _i ^- = b_i ^-$. It follows that for any Borel set $I \subset \bm{T}$, $ [\Lambda _U (I) ] _{ij}
:= [ \left( \chi _I (Z (U) ) b _i ^- , b _j ^- \right) _\Theta ] =  [\left(  \chi _I (Z_U ) \hat{W} _U ^* b_i ^- ,  \hat{W} _U ^* b_j ^- \right) _{\Theta _U} ] =
[\left( \chi _I (Z_U ) \beta _i ^- , \beta _j ^- \right) _{\Theta _U} ] = [\Om _U ] _{ij}$,
where the last equality follows from the fact that $Z_U$ is multiplication by the independent variable in $L^2 _{\Theta _U}$.

Note that if $\Theta $ is extreme so that $L^2 _\Theta = H^2 _\Theta$, then the above argument simplifies. In particular in this case
$Z^* _U = Y^* _U$ and we have no need to use dilation theory.

\end{proof}

\subsubsection{Remark} By the proof of the above proposition, $Z_{\Theta _U}, Z_\Theta (U)$ are the minimal unitary dilations
of $Y_{\Theta _U}$ and $Y_\Theta (U)$, respectively. By Lemma \ref{lemming:intertwiner}, there is a unitary operator $W_U$ intertwining
$Y_{\Theta _U}$ and $Y_\Theta (U)$. The above proof shows that there is a unitary $\hat{W} _U : L^2 _\Theta \rightarrow L^2 _{\Theta _U}$ which intertwines $Z_\Theta (U)$ and $Z_{\Theta _U}$ and satisfies $\hat{W} _U | _{H^2 _\Theta} = W_U$.
\label{subsubsection:equiv}

Recall that the earlier Proposition \ref{prop:acmeasure} established the more general statement that $\Om _{\Theta _A} = \La _A$ for any
$A \in  \ov{ \left( \bm{M} _n \right) _1 }$. This more general fact is not needed to prove the disintegration theorem. In the proof
of the disintegration theorem, one simply needs to show that $\Om _{\Theta _U} = \La _U $ for $U \in \mc{U} (n)$, as shown in
the above proposition, Proposition \ref{prop:weakprop}, as well as the fact that $\Om _0 = \La _0 = m$. Below we provide
a proof of this fact which does not rely on the methods of Subsection \ref{subsection:ACmeas}, so that the disintegration theorem, Theorem \ref{thm:disintegrate} as given in Subsection \ref{subsection:Weylint} can be proven completely using the results of this section
instead of those of Subsection \ref{subsection:ACmeas}.

\begin{lemming}
    $\Om _0 = \La _0 =m$.
\end{lemming}


\begin{proof}
    That $\Om _0 =m$ follows from the uniqueness of the Herglotz representation as described before the statement
of Theorem \ref{thm:disintegrate}.

By definition $\int _\bm{T} \zeta ^k [ \Lambda _0 (d\zeta ) ] _{ij}$ evaluates to $\left( Z (0) ^k b_i ^- , b_j ^- \right) _\Theta$
if $k\geq 0$ and to $\left( ( Z(0) ^* )^k b_i ^- , b_j ^- \right) _\Theta $ if $k \leq 0$.  The only non-vanishing moment occurs when
$k=0$ in which case this evaluates to $( b ^- _i , b ^- _j ) _\Theta
= \delta _{ij}$. This proves that $m = \La _0$ since they have the same moments.
\end{proof}

\section{Total orthogonal sets of point evaluation vectors}

\label{section:rkhssamp}

    If $\Theta$ is scalar-valued, necessary and sufficient conditions for the point evaluation vectors
$\delta _\zeta (z) := \frac{1 - \Theta (z) \ov{\Theta (\zeta)}}{1-z\ov{\zeta}} $ to belong to $K^2 _\Theta$
in the case where $\zeta \in \bm{T}$ can be given in terms of the existence of the Carath$\mr{\acute{e}}$odory
angular derivative (CAD) of $\Theta$ at $\zeta$ \cite[VI-4]{Sarason-dB}.  In \cite{Clark-perturb} (for inner $\Theta$)
and \cite{Fricain}, it is shown that $K^2 _\Theta$ has a total orthogonal set of point evaluation
vectors if and only if there is a $\zeta \in \bm{T}$ for which the measure $\Om _{\Theta _\zeta}$ is purely atomic. It is easy
to show that if $\{ \delta _{\la _n} \} _{n \in \bm{Z}}$ is a total orthogonal set in $K^2 _\Theta$, then $\{ \la _n \} \subset \bm{T}$.

This section will verify that these results generalize straightforwardly to the case where $\Theta $ is matrix-valued. To
accomplish this, it will first be useful to show how the theorems of \cite[Chapter VI]{Sarason-dB} on angular
derivatives extend to the matrix-valued case.

\subsection{Caratheodory angular derivatives}

    Let $\Theta$ be purely contractive. There is no need to assume that $\Theta (0) = 0$ in this subsection. The analytic
function $\Theta$ is said to have a Carath$\mr{\acute{e}}$odory angular derivative (CAD) at $\zeta \in \bm{T}$ if $\Theta$ has
a non-tangential limit $\Theta (\zeta )$ at $\zeta$, $| \Theta (\zeta ) | =1$, and the non-tangential limit of $\Theta '$
at $\zeta $ exists. In this case the CAD of $\Theta $ at $\zeta$ is defined as the limit of $\Theta ' (z)$ as $z\rightarrow \zeta$
non-tangentially, and is denoted by $\Theta ' (\zeta )$.

It is fairly easy to generalize \cite[VI-4]{Sarason-dB} to prove the following:

\begin{thm}
    If $\Theta \in \left( H^\infty _{\bm{M} _n} (\bm{D} ) \right) _1$ and $\zeta
\in \bm{T}$, the following are equivalent:
\bn
    \item $\Theta$ has a CAD at $\zeta \in \bm{T}$.
    \item $c _\zeta := \lim \inf _{z \stackrel{nt}{\rightarrow} \zeta} \| \frac{\bm{1} - \Theta (z) \Theta (\zeta ) ^* }{1-|z| ^2} \| < \infty$.
    \item There is a $U \in \mc{U} (n)$ such that $\frac{\Theta (z) - U}{z - \zeta} \vec{x} \in K^2 _\Theta$ for
    all $\vec{x} \in \bm{C} ^n$.
    \item Every element of $K^2 _\Theta$ has a non-tangential limit at $\zeta$.
\en
    If the above conditions hold then $\delta _\zeta ^{\vec{x}} \in K^2 _\Theta$ for any $\vec{x} \in \bm{C} ^n$,
if $f \in K^2 _\Theta$ then $\ip{f}{\delta_\zeta ^{\vec{x}}} _\Theta = (f (z) , \vec{x} )$ and $\delta _\zeta ^{\vec{x}}$ is the norm limit of $\delta_z ^{\vec{x}}$ as $z \stackrel{nt}{\rightarrow} \zeta$.
Moreover, $\Theta ' (\zeta ) = \ov{\zeta} A \Theta (\zeta )$ where $A > 0$
(so that $\Theta ' (\zeta )$ is invertible) and $\frac{ \bm{1} - \Theta (z) \Theta (z) ^*}{1 - |z| ^2 }$ converges
to $A$ as $z$ approaches $\zeta $ non-tangentially.

\label{thm:angderv}
\end{thm}

    In the above $z \stackrel{nt}{\rightarrow} \zeta$ denotes the non-tangential convergence of $z \in \bm{D}$ to
$\zeta \in \bm{T}$. The above theorem can be proven by following the proof for the case of scalar $\Theta$. The only part of the
proof which could be considered slightly more complicated is the proof that if (3) holds, then $\delta_z ^{\vec{x}} $
converges to $\delta _\zeta ^{\vec{x}}$ weakly, which is used in the proof that (3) $\Rightarrow$ (4).
We will show how this is accomplished and omit the rest of the proof.

As in the proof of \cite[VI-4]{Sarason-dB}, to show that $\delta_z ^{\vec{x}}$ converges weakly to $\delta _\zeta ^{\vec{x}}$
it suffices to show that the functions $\delta_z ^{\vec{x}}$ are bounded in norm as $z$ approaches $\zeta$ non-tangentially.
To show this it suffices to show that $\Delta _z$ is bounded in norm in this limit where $\Delta _w : \bm{C}^n
\rightarrow K^2 _\Theta$ is the linear map defined by $\Delta _w \vec{x} (z) := \Delta _w (z) \vec{x} = \delta_w ^{\vec{x}} (z)$.
This follows from an argument that can be found in the proof of \cite[Lemma 8.3]{Dym}: Consider \ba 0 & \leq & \left( (\bm{1} - \Theta (z) \Theta (z) ^* ) \vec{x} , \vec{x} \right) + \left(
(\Theta (\zeta ) ^* - \Theta (z) ^* ) \vec{x} , (\Theta (\zeta ) ^* - \Theta (z) ^* ) \vec{x} \right) \nonumber \\
& = & \left( (\bm{1} - \Theta (z) \Theta (\zeta ) ^* ) \vec{x} , \vec{x} \right) +  \left( \vec{x}, (\bm{1} - \Theta (z) \Theta (\zeta ) ^* ) \vec{x} \right) ,\ea and observe that both terms on the right hand side of the inequality on the first line are positive.
Recall that $\Delta _z (w) = \frac{\bm{1} - \Theta (w) \Theta (z) ^*}{1-w\ov{z}}$.
It follows that
\ba \| \Delta _z \vec{x} \| ^2 _\Theta & \leq & \frac{1 - z \ov{\zeta}}{1 - | z | ^2 } \ip{\Delta _\zeta \vec{x}}{\Delta _z \vec{x}} _\Theta
+ \frac{1 -\ov{z} \zeta}{1-|z| ^2} \ip{\Delta _z \vec{x}}{\Delta _\zeta \vec{x} } _\Theta \nonumber \\
& \leq &  2 \frac{| 1- z \ov{\zeta} |}{1 - |z| ^2} \| \Delta _\zeta \vec{x} \| _\Theta \| \Delta _z \vec{x} \| _\Theta .\ea This inequality
shows that $\Delta _z $ is bounded in norm as $z$ approaches $\zeta$ non-tangentially.

\subsubsection{Remark} More generally, given $\vec{x} \in \bm{C} ^n$ we will say that $\Theta \vec{x}$ has a CAD at $\zeta \in \bm{T}$ if $\Theta \vec{x}$ has
a non-tangential limit $\Theta (\zeta ) \vec{x}$ at $\zeta$, $\| \Theta (\zeta ) \vec{x} \| = \| \vec{x} \| $, and $\Theta ' \vec{x}$ has
a non-tangential limit at $\zeta$. One can prove a version of the above theorem for such vector functions. We will
not write this result down here, but we note that one can show that $\delta_\zeta ^{\vec{x}} \in K^2 _\Theta$ if and only if there is a unitary $U$ such that
$\Theta U^* \vec{x}$ has a CAD at $\zeta$. \label{subsubsection:vecangderv}

\subsection{Spectra of the unitary perturbations $Z _\Theta (U)$}

\label{subsection:spectra}

    Earlier we defined $Z_\Theta ' := Z_\Theta | _{L^2 _\Theta \ominus \mf{D} _+}$. This is clearly an isometric linear transformation
from $L^2 _\Theta \ominus \mf{D} _+$ onto $L^2 _\Theta \ominus \mf{D} _-$.
The deficiency indices of an isometric linear transformation $V$ are defined as $(n_+ , n_-)$ where $n_+ := \dim{ \dom{V} ^\perp}$
and $n_- := \dim{\ran{V} ^\perp}$. If $\Theta$ has rank $n$, it follows that the deficiency indices of $Z_\Theta '$ are $(n,n)$.
An isometric linear transformation is called simple if it has no unitary restriction to a proper subspace. It is easy to see
that $Z_\Theta '$ is simple, as if it were not, then $Z_\Theta $ would have a reducing subspace orthogonal to $\mf{D} _- = \{ 1/z e_i \}$,
which, as discussed at the beginning of Section \ref{subsection:ACmeas} is not possible. A point $\la \in \bm{C} $ is called regular for an isometric linear transformation $V$ if $ V-\la$ is bounded below. $V$ is
called regular if every $\la \in \bm{C} \sm \{ 1 \}$ is regular for $V$ (\emph{i.e.} if every $\la \in \bm{C}$ is regular for the symmetric
linear transformation $S = \mu ^{-1} (V)$ defined on $\ran{V-\bm{1}}$) where $\mu (z) = \frac{z-i}{z+i}$.

As proven by Lifschitz in \cite{Lifschitz2}, any simple isometric linear transformation $V$ with indices $(n,n)$ is unitarily equivalent
to $Z_\Theta '$ for some purely contractive $\Theta$ with $\Theta (0) =0$. The following theorem characterizes the essential spectrum of $Z_\Theta '$ (and hence of $Z_\Theta$) \cite[Theorem 4]{Lifschitz2}

\begin{thm}{ (Lifschitz)}
    A point $\zeta \in \bm{T}$ is a regular point of $Z_\Theta '$ if and only if both of the following conditions are satisfied:
    \bn
        \item $\Theta$ is analytic on some open neighbourhood of $\zeta$.
        \item There is a neighbourhood $N_\zeta$ of $\zeta$ such that $\Theta (\la)$ is
        unitary for all $\la \in N_\zeta \cap \bm{T}$.
    \en

    \label{thm:esspec}
\end{thm}

    By the above theorem the essential spectrum, $\sigma _e (Z_\Theta (U))$ of any of the unitary perturbations $Z_\Theta (U)$
is the set of all $\zeta \in \bm{T}$ which fail to satisfy at least one of the above conditions in the theorem. We will denote this
set by $\mr{sp} (\Theta)$. Assume that $\zeta \in \bm{T} \sm \mr{sp} (\Theta )$. Then $Z_\Theta (U) - \zeta $ is a finite
rank perturbation of $Z_\Theta (0) -\zeta$ which has Fredholm index $0$ since both $Z_\Theta (0)$ and its adjoint are simple.
It follows that $\sigma (Z_\Theta (U)) = \mr{sp} (\Theta) \cup
\sigma _p (Z_\Theta (U))$, where $\sigma _p (Z_\Theta (U))$ is the set of eigenvalues of $Z_\Theta (U)$. To determine
the spectrum of $Z_\Theta (U)$ it remains to determine its eigenvalues.

    It is worth noting that one can show using the basic theory of isometric/symmetric linear transformations that given a simple
isometric linear transformation $V$ with deficiency indices $(n,n)$, any eigenvalue of any unitary extension $U$ of $V$ has multiplicity
not exceeding $n$, and if $\vec{\la}$ is any point in $\bm{T} ^n$ consisting of regular points for $V$, there is a unitary extension $U$ of $V$ which has the entries of $\vec{\la}$ as eigenvalues. Moreover each distinct pair of unitary extensions $V(U)$ and $V(U')$ can share no more than $n-1$ eigenvectors.
See for example \cite[Section 83]{Glazman}

\begin{prop}
    Suppose that $\Theta (0) =0$ and $\la \in \bm{T}$.
    \bn
        \item $\la \in \sigma _p (Z_\Theta (U) ) \sm \mr{sp} (\Theta )$ if and only if $\ker{\Theta (\la ) ^* - U^*} \neq \emptyset$.
         A vector $\vec{x} \in \bm{C} ^n$ belongs to $ \ker{\Theta (\la) ^* - U^*}$ if and only if $\delta _{\{ \la \} } \vec{x}$ is an eigenvector of $Z_{\Theta _U}$ to eigenvalue $\la$.

        \item $\la$ is not an eigenvalue of any $Z_\Theta (U)$ if and only if $\lim _{z \stackrel{nt}{\rightarrow} \la } (1 -z\ov{\la}) U (U - \Theta (z) ) ^{-1} = 0$. This happens if and only if the angular derivative of $\Theta \vec{x}$ at $\la$ does not exist for any $\vec{x} \in \bm{C} ^n$.
    \en

    \label{prop:evalues}
\end{prop}

    In the above $\delta _{ \{ \la \} } \vec{x} \in L^2 _\Theta$ is the point mass function which takes the
value $\vec{x}$ at $\la \in \bm{T}$ and vanishes elsewhere on $\bm{T}$.

\begin{proof}
    By Remark \ref{subsubsection:equiv}, $Z_\Theta (U)$ is unitarily equivalent to $Z_{\Theta _U}$ which acts as multiplication
by $z$ in $L^2 _{\Theta _U}$. It follows that $\la \in \bm{T}$ is an eigenvalue of $Z_\Theta (U)$ if and only if $\Om _{\Theta _U}$
has a point mass at $\la$, \emph{i.e.} if and only if $\Om _{\Theta _U} ( \{ \la \} ) \neq 0$.

Now by the Herglotz theorem \be 2 (\bm{1} -\Theta (z) U^*) ^{-1} = B _{\Theta _U} (z) + \bm{1} = 2 \int _\bm{T} \frac{1}{1 -\ov{\zeta}z} \Om _{\Theta _U} (d\zeta), \ee and note that $(\bm{1}  -\Theta (z) U^* ) ^{-1} = U ( U -\Theta (z) ) ^{-1}$.  It follows easily from this that
\be \Om _{\Theta _U} [ \{ \la \} ] = \lim _{z \rightarrow \la} (1 - z \ov{\la} ) U (U -\Theta (z) ) ^{-1}. \label{eq:projform} \ee
In the above limit, we
assume $z$ converges to $\la$ non-tangentially.  Hence $\la $ is not an eigenvalue of $Z _\Theta (U)$ if and only if this limit is identically $0$. This happens if and only if
\be \lim _{z \rightarrow \la}  \| \frac{ \left( \Theta (z) - U \right) }{z - \la} \vec{x} \| = \infty ,\ee for every $\vec{x}
\in \bm{C} ^n$. This shows that $\la$ is not an eigenvalue of any $Z_\Theta (U)$, $U \in \mc{U} (n)$ if and only if
the angular derivative of $\Theta (z) \vec{x}$ at $\la $ does not exist for any $\vec{x} \in \bm{C} ^n$ (see Remark \ref{subsubsection:vecangderv}).

Since $Z_{\Theta _U}$ acts as multiplication by $z$, clearly $\la $ is an eigenvalue of
$Z_{\Theta _U}$ if and only if there is a $\vec{x} \in \bm{C} ^n$ such that $\delta _{ \{ \la \} } \vec{x}$ is an eigenvector
of $Z_{\Theta _U}$. If $\delta _{\{ \la \} } \vec{x}$ is such an eigenvector, then $\vec{x} \in \bm{C} ^n$ must be in the range of the
non-zero projection $\Om _{\Theta _U} [ \{ \la \} ] \in \bm{M} _n (\bm{C})$. Hence,
\be \vec{x} = \Om _{\Theta _U} [ \{ \la \} ] \vec{x} = \lim _{z \rightarrow \la} (1 - z \ov{\la})(\bm{1} - \Theta(z)U^*)^{-1} \vec{x}.\ee
This in turn implies that $\lim _{z \rightarrow \la} (\bm{1} - \Theta(z) U^*) \vec{x} =0$ so that $(\Theta (\la) ^* - U^*) \vec{x} =0$.

Conversely suppose that $\vec{x} \in \ker{\Theta (\la) ^* - U^*}$. If $\la \notin \mr{sp} (\Theta)$, it follows from Theorem \ref{thm:esspec}, that $\Theta$ is analytic in a neighbourhood of $\la$, so that in particular the angular derivative of $\Theta $ exists at $\la$. By Theorem
\ref{thm:angderv}, the angular derivative $\Theta ' (\la)$ is invertible, and it is the limit of the invertible matrices
$A(z) := \frac{\Theta (z) - \Theta (\la)}{z-\la}$ as $z \rightarrow \la$ non-tangentially.

Recall the matrix analytic function $\Delta _\la (z) := \frac{\bm{1} - \Theta (z) \Theta (\la ) ^*}{1 - z \ov{\la}}$. By
Theorem \ref{thm:angderv}, $\Delta _\la (z)$ converges to $ \Delta _\la (\la) := \la \Theta (\la ) ^* \Theta ' (\la)$ as $z \rightarrow \la$
non-tangentially, and this limit is an invertible operator. The non-tangential limit of $\Delta _\la (z) ^{-1}$ at $\la$
is equal to the projection $\Om _{\Theta _{\Theta (\la ) }} [ \{ \la \} ]$ by equation (\ref{eq:projform}) so that
$\Delta _\la (z) ^{-1} $ is norm bounded in this limit and the non-tangential limit of $\Delta _\la (z) ^{-1}$ is equal to $\Delta _\la (\la) ^{-1}$. Since this is an invertible
projection, $\Delta _\la (\la) = \Delta _\la (\la) ^{-1} = \bm{1}$.

Let $B(z) := \frac{\bm{1} - \Theta (z) U^*}{1-z\ov{\la}}$. Previous calculations in this proof
have shown that $B(z) ^{-1} \rightarrow \Om _{\Theta _U} [ \{ \la \} ] \vec{x}$. To show that $\delta _{ \{ \la \} } \vec{x}$
is an eigenvector of $Z_{\Theta _U}$, we need to show that $B(z) ^{-1} \vec{x}$ converges to $\vec{x}$
as $z \rightarrow \la $ non-tangentially. This is easily accomplished by observing that $\| B(z) ^{-1} \vec{x} - \vec{x} \|
\leq \| B(z) ^{-1} \| \| \vec{x} - B(z) \vec{x} \| = \| B(z) ^{-1} \| \| \vec{x} - \Delta _\la (z) \vec{x} \|$.  The last
equality follows from the fact that $\vec{x} \in \ker{\Theta(\la) ^* - U^*}$. Since $\| B(z) ^{-1} \| $ is bounded as
$z \rightarrow \la $ non-tangentially, and $\Delta _\la (z)$ converges to $\bm{1}$, the proof is complete.
\end{proof}

\subsection{Total orthogonal sets of point evaluation vectors}

    Recall the matrix kernel functions $\Delta _w (z) := \frac{\bm{1} -\Theta (z) \Theta (w) ^*}{1-z\ov{w}}$,
and the point evaluation functions $\delta ^{\vec{x}} _w := \Delta _w \vec{x} \in K^2 _\Theta$ which satisfy $\ip{f}{\delta _w ^{\vec{x}}} _\Theta
= (f (w) , \vec{x} )$ for all $f \in K^2 _\Theta$, all $w \in \bm{D}$, and all $w \in \bm{T}$ for which the
angular derivative of $\Theta$ at $w$ exists.

In this section we determine necessary and sufficient conditions for $K^2 _\Theta$ to have a total orthogonal set
of point evaluation functions. Suppose that $\La := \{ \delta _{\la _i } ^{\vec{x} _i } \} _{i \in \bm{Z}} \subset K^2 _\Theta $
is such a set.  For convenience define $\delta _i := \delta _{\la _i} ^{\vec{x} _i}$.
Let $N _\La :K^2 _\Theta \rightarrow K^2 _\Theta$ be the normal operator $N_\La := \sum _{n \in \bm{Z}}
\la _i \frac{\ip{\cdot}{\delta _i } \delta _i}{\| \delta _i \| ^2 }$.

\begin{prop}
    If $\La = \{ \delta _i \} _{i \in \bm{Z}} \subset K^2 _\Theta$ is a total orthogonal set, then
$\Theta$ is extreme and $N_\La$ is unitarily equivalent to $Z_\Theta (U)$ for some $U \in \mc{U} (n)$.
Hence $N_\La$ is unitary and $\{ \la _i \} _{i \in \bm{Z}} \subset \bm{T}$. \label{prop:tot}
\end{prop}

The following simple fact will be used in the proof of the above proposition.

\begin{lemming}
    Let $V$ be an isometric linear transformation with deficiency indices $(n,n)$, and let $P, Q$ be the projectors
onto $\dom{V}$ and $\ran{V}$ respectively. If $(P \vee Q ) ^\perp =0$ then any normal extension of
$V$ must be unitary. \label{lemming:unitary}
\end{lemming}

\begin{proof}
    Given any $\phi \in \mc{H}$ there exist $\phi _1 \in Q\mc{H}$ and $\phi _2 \in P \mc{H}$ such that $\phi = \phi _1 + \phi _2$.
Any extension of $V$ can be written as $V(A) = V \oplus A $ on $\mc{H} = \dom{V} \oplus \dom{V} ^\perp$ where $A : \dom{V} ^\perp
\rightarrow \ran{V} ^\perp$. If $V(A)$ is a normal extension of $V$ then $V(A) ^* V(A) = P +  A^* A = V(A) V(A) ^* = Q + A A^*$
where $A^* A$ vanishes on $\dom{V}$ and $AA^*$ vanishes on $\ran{V}$.

It follows that $V(A) ^* V(A) \phi = V(A) ^* V(A) \phi _1 + V(A) V(A) ^* \phi _2 = P \phi _1 + Q \phi _2 =
\phi _1 + \phi _2 = \phi$. Hence $V(A)$ is unitary.
\end{proof}

\begin{lemming}
    Let $P$ be the projector onto $\mf{D} _- \subset H^2 _\Theta$ and let $Q:= V_\Theta P V_\Theta ^*$.
The restriction $Y_\Theta (0) = Z_\Theta (0) | _{H^2 _\Theta}$ and $X_\Theta$, the restriction of the backwards shift
to $K^2 _\Theta$ are related by the following formula:

\be V_\Theta Y_\Theta (0) ^* V_\Theta ^* = X_\Theta (\bm{1} -Q).\ee
\label{lemming:balls}
\end{lemming}

\begin{proof}
    Let $R_\Theta$ be the projection of $L^2 _\Theta$ onto $H^2 _\Theta$ so that $Y_\Theta (0) ^*
= R_\Theta Z_\Theta (0) ^* R_\Theta = Y_\Theta ^* (\bm{1} - P)$. Then \be V_\Theta Y_\Theta (0) ^* V_\Theta ^*
= V_\Theta Y_\Theta ^* V_\Theta ^* V_\Theta (\bm{1} -P ) V_\Theta ^* = V_\Theta Y_\Theta ^* V_\Theta ^* (\bm{1} -Q).
\label{eq:balls2}\ee
By equations (\ref{eq:opex}) and (\ref{eq:actonconst}),
\ba V_\Theta Y_\Theta ^* V_\Theta ^* & = & X_\Theta + V_\Theta Y_\Theta ^* P (\bm{1} - \Theta (0) ) PV_\Theta ^*
\nonumber \\ & =& X_\Theta + (S^* \Theta) P (\bm{1} - \Theta (0) ) ^{-1} (\bm{1} - \Theta (0)) P V_\Theta ^* \nonumber
\\ & = & X_\Theta + (S^* \Theta) P V_\Theta ^*. \ea In the above note that any $A \in \bm{M} _n (\bm{C})$ is viewed as
the operator $\sum _{ij} ( \cdot , b_i ^- ) _\Theta A _{ij} b_j ^-$ where $\{ b_i ^- \}$ is the fixed
ON basis of $\mf{D} _-$ so that in particular $P \Theta (0) P = \Theta (0) P = \Theta (0)$. Equation (\ref{eq:balls2}) becomes
\be V_\Theta Y_\Theta (0) ^* V_\Theta ^* = (X_\Theta  + (S^* \Theta) P V _\Theta ^*) (\bm{1} - Q) = X_\Theta (\bm{1} -Q).\ee
\end{proof}

\begin{proof}{ (of Proposition \ref{prop:tot})}
    Recall the canonical unitary transformation $V_\Theta : H^2 _\Theta \rightarrow K^2 _\Theta$ from the Cauchy integral
representation of $K^2 _\Theta$. Let $P=P_-$ be the projector onto the constant functions in $H^2 _\Theta$ spanned by the
basis $\{ e_i  \} _{i=1} ^n$ and let $Q = V_\Theta P V_\Theta ^*$ be the projector in $K^2 _\Theta$ onto the span
of the vectors $\delta _0 ^{e _i} = \Delta _0 e_i$ for $1 \leq i \leq  n$ (if $\Theta (0) = 0$ these are constant functions).
If $A \in \ov{(\bm{M} _n)}_1 $ then
$Z_\Theta (A) ^* = Z_\Theta ^* ( \bm{1} + (A - \bm{1} ) P )$. Let $P_\Theta $ denote the projector onto $H^2 _\Theta$.
Then $Y_\Theta (A) ^* = P_\Theta Z_\Theta (A) ^* P _\Theta = P_\Theta Z_\Theta ^* P_\Theta (\bm{1} + (A - \bm{1} ) P)$.

By Lemma \ref{lemming:balls}, $V_\Theta Y_\Theta ^* (0) V_\Theta ^* = X_\Theta (\bm{1} -Q)$. Now if
$f \in K^2 _\Theta \ominus Q K^2 _\Theta$,
then since $0 = \ip{f}{\delta _{0} ^{b_i} } _\Theta$ for all $1\leq i \leq n$, it follows that $f(0) =0$. Hence
$f(z) = z g(z)$ for some $g \in H^2 _n (\bm{D})$. Moreover since $K^2 _\Theta$ is invariant for $S^*$, it follows that
$S^*f = g \in K^2 _\Theta$. Hence for any $\vec{x} \in \bm{C} ^n$ and any $\la \in \bm{D}$ or $\la \in \bm{T}$ for which
the angular derivative of $\Theta $ at $\la$ exists,
\be \ip{X_\Theta f}{\delta _\la ^{\vec{x}}} _\Theta = ( g(\la ) , \vec{x} )  = \ov{\la} \ip{f}{\delta _\la ^{\vec{x}}} _\Theta.\ee
It follows that
\be X_\Theta (\bm{1} -Q) f = X_\Theta f = \sum _{n \in \bm{Z}} \ov{\la} _n \frac{\ip{f}{\delta _n} _\Theta \delta _n}{\| \delta _n \| ^2 }
= N^* f .\ee It can be concluded that $N^*$ is a normal extension of $X_\Theta | _{Q ^{\perp} K^2 _\Theta}$, so that $\hat{N} := V_\Theta ^* N^* V_\Theta$ is a normal and contractive extension of $Y_\Theta (0) ^* | _{P ^\perp L^2 _\Theta}$.

Now $Y_\Theta ^* $ is a co-isometry, and by the Wold decomposition it can be decomposed into the direct sum of a unitary operator,
and a purely co-isometric operator (an operator isomorphic to the direct sum of copies of the adjoint of the unilateral shift). If $Y_\Theta ^*$ had
a non-zero purely co-isometric part, then it would have non-zero Fredholm index. However, $\hat{N} ^*$ is a normal finite rank perturbation of $Y_\Theta ^*$, and hence is Fredholm. Any normal Fredholm operator must have index zero. Since the index is invariant under
compact perturbations, $Y_\Theta ^*$ also has index $0$ and hence $Y_\Theta ^*$ is unitary. It follows that $H^2 _\Theta= L^2 _\Theta$,
that $Y_\Theta = Z_\Theta$ and that $\Theta$ is extreme.

In conclusion $\hat{N} ^*$ is a normal extension $Z _\Theta (0) ^*$ which is a partial isometry with deficiency indices $(n,n)$. By
Lemma \ref{lemming:unitary}, $\hat{N} ^*$ and hence $\hat{N}$ must be unitary so that $\hat{N} = Z_\Theta (U) $ for some $U \in \mc{U} (n)$.
Since $\hat{N}$ is unitary its spectrum is contained in the unit circle so that $\{ \la _n \} \subset \bm{T}$.
\end{proof}

\begin{thm}
    $K^2 _\Theta$ has a total orthogonal set of point evaluation vectors if and only if there is a $U \in \mc{U} (n)$
such that the measure $\Om _U$ is purely atomic. If $K^2 _\Theta$ has such a set then $\Theta $ is inner.
\end{thm}

\begin{proof}
    If $K^2 _\Theta$ has a total orthogonal set of point evaluation vectors $\{ \delta _i \}$, where $\delta _i = \delta _{\la _i}
^{\vec{x} _i}$ then by the
previous proposition, there is a $U \in \mc{U} (n)$ such that $Z _\Theta (U)$ has a total orthogonal set of
eigenfunctions. Therefore $Z_{\Theta _U}$ which acts as multiplication by $z$ in $H^2 _{\Theta _U} = L^2 _{\Theta _U}$
has $\{ \delta _{ \{ \la _i \}  } \vec{x} _i \}$ as a total orthogonal set of eigenfunctions, and the measure
$\Om _{\Theta _U} = \Om _U = \sum _{n \in \bm{Z}} \Om _U ( \{ \la _i \} ) \delta _{ \{ \la _i \} } $ is purely atomic.

Conversely if $\Om _U  = \sum _{n \in \bm{Z}} \Om _U ( \{ \la _i \} ) \delta _{ \{ \la _i \} } $ is purely atomic then
$\{ \delta _{ \{ \la _i \} } \vec{x} _i ^j \} _{1 \leq j \leq k_i ; \ i \in \bm{Z}}$ where $\{ \vec{x} _i ^j \} _{j=1} ^{k_i}$
is an ON basis for $\Om _U \left( \{ \la _i \} \right) \bm{C} ^n$, and $k_i \leq n$, is a total orthogonal set of eigenvectors to $Z_{\Theta _U}$.
Note here that each $\Om _U \left( \{ \la _i \} \right)$ is a projection.
Under the canonical unitary transformation $V_{\Theta _U} : H^2 _{\Theta _U} \rightarrow K^2 _\Theta$,
\be V_{\Theta _U} \delta _{ \{ \la _i \} } \vec{x} _i ^j (z) = (1 -\Theta _U (z)) \int _{\bm{T}} \frac{ \delta _{ \{\la _i \} } (w) }{ 1 - z\ov{w}}
\Om _{\Theta _U} (dw ) \cdot \vec{x} _i ^j =  \frac{\bm{1} - \Theta (z) U^*}{1 - z \ov{\la _i} } \vec{x} _i ^j.\ee By Proposition \ref{prop:evalues},
$\vec{x} _i ^j \in \ker{\Theta (\la _i) ^* -U^*}$, so that $V_{\Theta _U} \delta _{\{ \la _i \} } \vec{x} _i ^j = \delta _{\la _i } ^{\vec{x} _i ^j}$. We conclude that
$\{ \delta _i \}$ where $\delta _i = \delta _{\la _i } ^{\vec{x} _i ^j }$ is a total orthogonal set of point evaluation
vectors in $K^2 _\Theta$.

If $\Theta $ is not inner, then there is a set $I \subset \mr{Bor} (\bm{T})$ with $m ( I ) > 0$ such that $\Theta (z)$ is not
unitary for $z \in I$. Let $\Om _{\Theta _U} ^a$ denote the absolutely continuous part of $\Om _{\Theta _U}$ with respect
to $m$. Then \be \left( \bm{1} - \Theta (z) U^* \right) ^{-1} \left( \bm{1} -\Theta (z) \Theta (z) ^* \right) \left(  \bm{1} - U \Theta (z)^*  \right) ^{-1} =
\int _\bm{T} \frac{1 - |z| ^2 }{|1 -z \ov{\zeta }  | ^2 } \Om _{\Theta _U} (d\zeta ) ,\ee for $z \in \bm{D}$, and
\be \frac{d\Om _{\Theta _U} ^a }{dm} (\zeta) = \left( \bm{1} - \Theta (\zeta ) U^* \right) ^{-1} \left( \bm{1} -\Theta (\zeta) \Theta (\zeta) ^* \right) \left(  \bm{1} - U \Theta (\zeta )^*  \right) ^{-1} ,\ee almost everywhere $\zeta \in \bm{T}$ with respect to Lebesgue measure. For a proof of this fact in the matrix setting, see \cite[Theorem 9]{Elliott}.
Hence if $\Theta$ is not inner, $\Om _{\Theta _U}$ cannot be purely atomic for any $U \in \mc{U} (n)$ so that $K^2 _\Theta$ cannot have a total orthogonal set of point evaluation
vectors.
\end{proof}

\section{Representation of simple symmetric operators with deficiency indices $(n,n)$}

\label{section:symrep}

    In this final section, we wish to point out that any simple symmetric operator with deficiency
indices $(n,n)$ is unitarily equivalent to the symmetric operator of multiplication by the independent variable
in a model subspace $K^2 _\Phi$ where $\Phi \in H^\infty _{\bm{M} _n} (\bm{U})$ is inner, $\Phi (i) =0$ and $\Phi$
is analytic on some open neighbourhood of any given point $x \in \bm{R}$.  We will see that such $K^2 _\Phi$ have a
$\mc{U} (n) -$parameter family of total orthogonal sets of point evaluation vectors. Recall that $\bm{U}$ denotes
the open upper half plane, and $H^\infty _{\bm{M} _n} (\bm{U})$ is the Hardy space of bounded analytic $\bm{M} _n$-valued
functions on $\bm{U}$.

There is a bijective correspondence between $\Phi \in H^\infty _{\bm{M} _n} (\bm{U} )$ and $\Theta \in H^\infty _{\bm{M} _n}
(\bm{D} )$ given by $\Phi = \Theta \circ \mu $ and $\Theta = \Phi \circ \mu ^{-1}$ where $\mu (z) = \frac{z-i}{z+i}$ and
$\mu ^{-1} (z) = i \frac{1+z}{1-z}$. Further recall that there is a canonical unitary transformation $\mc{U} : H^2 _n (\bm{D})
\rightarrow H^2 _n (\bm{U} )$ given by
\be \mc{U} f(z) = \frac{1 - \mu (z)}{\sqrt{\pi}} f \circ \mu (z) ,\ee and that $\mc{U}$ takes $K ^2 _\Theta $ onto $K^2 _\Phi$.

\subsection{Representation of simple symmetric linear transformations with deficiency indices $(n,n)$}

    Recall the Lifschitz characteristic
function of a simple isometric linear transformation $V$. Here $\dom{V}$ and $\ran{V}$ are contained in a separable
Hilbert space $\mc{H}$. Let $\mf{D} _+ := \dom{V} ^\perp$ and $\mf{D} _- := \ran{V} ^\perp$,
fix a unitary extension $U$ of $V$ and let $(\psi _i ^\pm ) _{i=1} ^n$ be orthonormal bases of $\mf{D} _\pm$ such that
$U \psi _i ^+ = \psi _i ^-$. Given $W \in \mc{U} (n)$, we define $ V(W) := V \oplus \sum _{ij} \ip{\cdot}{\psi _i ^+} W_{ij} \psi _i ^-$
on $\mc{H} := \dom{V} \oplus \mf{D} _+$, so that $\{ V(W) \} _{W \in \mc{U} (n)}$ is the $\mc{U} (n)$-parameter
family of unitary extensions of $V$.

\subsubsection{Definition} Fix $U \in \mc{U} (n)$. For $1\leq i,k \leq n$, let $A, B$ be matrix valued functions on $\bm{D}$ with entries $A_{ik} (z)
= z \ip{(V(U) -z) ^{-1} \psi _i ^+ }{ \psi _k ^+}$, $B_{ik} (z) := \ip{(V(U) -z) ^{-1} V(U) \psi _i ^+}{\psi _k ^+}$. The Lifschitz
characteristic function of the simple isometric linear transformation $V$ is defined as $\Theta _V (z) := A(z) B(z) ^{-1}$.

    One can show that $\Theta _V (z)$ is always a purely contractive matrix analytic function on $\bm{D}$ with $\Theta _V (0) = 0$. Two
contractive matrix analytic functions on $\bm{D}$, $\Theta _1$ and $\Theta _2$ are said to coincide if there are fixed unitaries $U,V$ in $\mc{U} (n)$
such that $U\Theta _1 = \Theta _2 V$. In \cite{Lifschitz2}, it is shown that two simple isometric linear transformations $V_1, V_2$ are unitarily
equivalent if and only if their characteristic functions coincide.  Moreover one can show that choosing a different $U$ in the definition
of $\Theta _V$ yields another purely contractive function which coincides with the original so that $\Theta _V$ is unique up to such
coincidence. Now given $\Theta _V$, consider the operator $Z_{\Theta _V}$ of multiplication by the independent variable in $L^2 _{\Theta _V}$.
As discussed in the beginning of Section \ref{subsection:spectra}, the transformation $Z_{\Theta _V} ' = Z_{\Theta _V} | _{L^2 _{\Theta_V} \ominus \mf{D} _+}$
is a simple isometric linear transformation with deficiency indices $(n,n)$. It is not difficult to show that the characteristic
function of $Z _{\Theta _V} ' $ is $\Theta _V $ so that $V$ is always unitarily equivalent to $Z _{\Theta _V}$.

\begin{thm}{ (Lifschitz)}
    Any simple isometric linear transformation $V$ with deficiency indices $(n,n)$ is unitarily equivalent to $Z_{\Theta _V} ' $,
which acts as multiplication by the independent variable on $\dom{Z _{\Theta _V} '} = L^2 _{\Theta _V} \ominus \mf{D} _+$. \label{thm:represent}
\end{thm}

\begin{proof}
    Let $\Theta := \Theta _V$. As in the proof of Proposition \ref{prop:evalues}, it is easy to check that
\be \left( (B_\Theta - \bm{1} ) e_i , e_k \right) = 2z \int _\bm{T} \frac{1}{\zeta - z} \left( \Om _\Theta (d\zeta) e_i , e_k \right)
= 2z \ip{ (Z_\Theta -z ) ^{-1} b _i ^+}{b _k ^+} _\Theta, \ee and similarly that
\be \left( (B_\Theta + \bm{1}) e_i , e_k \right) = 2 \ip{ Z_\Theta (Z_\Theta -z) ^{-1} b _i ^+}{b _k ^+} _\Theta. \ee
This shows that the Lifschitz characteristic function of $Z_\Theta '$ coincides with $\Theta$ since
$\Theta = (B_\Theta - \bm{1} ) (B_\Theta + \bm{1} ) ^{-1}$.
\end{proof}

    There is a bijective correspondence between simple isometric linear transformations $V$ and
simple symmetric linear transformations $B$ given by $B = \mu ^{-1} (V)$ with $\dom{B} := (V - \bm{1} ) \dom{V}$ and
$V = \mu (B)$ with $\dom{V} = \ran{B+i}$. Recall here that a symmetric linear transformation is called simple
if it has no self-adjoint restriction to a proper subspace.

The following provides necessary and sufficient conditions on $\Theta $ for the symmetric linear transformation
$\mu ^{-1} (Z ' _\Theta )$ to be a densely defined symmetric operator. This is a straightforward generalization
of a result of Lifschitz for the case $n=1$, and the proof is virtually identical.

\begin{lemming}
    Let $V$ be a simple isometric linear transformation with deficiency indices $(n,n)$. Then $B = \mu ^{-1} (V)$
is a densely defined symmetric operator if and only if $z=1$ is not an eigenvalue of any unitary extension of $V$.
\end{lemming}

\begin{proof}
    If $\mu ^{-1} (V)$ is densely defined, then $\ran{V - \bm{1}}$ is dense so that if $U$
is any unitary extension of $V$, then
$\ran{U -\bm{1}} \supset \ran {V - \bm{1}}$ is also dense. This can only happen if $z=1$ is not an eigenvalue
of any unitary extension $U$.

Conversely if $\mu ^{-1} (V)$ is not densely defined then there is a $\xi \in \mc{H}$ such that
$\ip{ (V- \bm{1} )\psi }{\xi } = 0 $ for all $\psi \in \dom{V}$. Let $\{ \psi _i ^+ \} $ and $\{ \psi _i ^- \}$
be ON bases for $ \dom{V} ^\perp $ and $\ran{V} ^\perp $ respectively. For $A \in \bm{M} _n$ define
$V (A) := V \oplus \hat{A}$ on $\mc{H} = \dom{V} \oplus \dom{V} ^\perp$ where $\hat{A} : \dom{V} ^\perp \rightarrow
\ran{V} ^\perp$ is given by $\hat{A} = \sum _{i,j =1} ^n A_{ij} \ip{\cdot}{\psi _i ^+} \psi _j ^-$.

Given any $\psi \in \mc{H} = \dom{V} \oplus \dom{V} ^\perp$, $\psi = \psi _V + \sum _{i=1} ^n c_i \psi _i ^+$ where $\psi _V \in \dom{V}$.
Hence, \be \ip{ (V(A) - \bm{1} ) \psi }{\xi} = \ip{(V-1) \psi _V }{\xi} + \sum _{i=1} ^n c_i
\ip{ (\hat{A} - \bm{1} ) \psi _i ^+ }{\xi}. \ee

Now $\xi$ is not orthogonal to $\ran{V} ^\perp$, as otherwise there would exist a
$\xi ' \in \dom{V}$ such that $V \xi ' = \xi $. This would imply that
\be 0 = \ip{(V-\bm{1} ) \psi }{V \xi ' } = \ip{\psi }{(1-V) \xi ' } ,\ee for all $\psi \in \dom{V}$
so that $(V - 1 ) \xi ' \in \dom{V} ^\perp$. The fact that $\xi ' $ is orthogonal to $\dom{V} ^\perp$
and that $\| V \xi ' \|  = \| \xi \| $ would then imply that $V \xi ' = \xi '$, contradicting the simplicity
of $V$. We conclude that $\xi $ is not orthogonal to $\ran{V} ^\perp $, so that we can choose $A$ so that
$\ip{(\hat{A} - \bm{1} ) \psi _i ^+ }{\xi } =0$.

It follows that for this choice of $A$, $\ip{(V(A) -\bm{1} ) \psi }{\xi} = \ip{(V-\bm{1}) \psi _V }{\xi} =0$ for
all $\psi \in \mc{H}$ so that $V(A)^* \xi = \xi$. Now $\xi = \xi _V ^* + \psi ^-$ where $\xi _V ^* \in \ran{V}$
and $\psi  ^- \in \ran{V} ^\perp$, and $V(A) ^* = V^* \oplus \hat{A} ^* $ on $\mc{H} = \ran{V} \oplus \ran{V} ^\perp$.
A simple calculation shows
\be \| \xi _V ^* \| ^2 + \| \psi ^- \| ^2 = \| \xi \| ^2 = \| V(A) ^* \xi \| ^2  = \| V ^* \xi _V ^* \| ^2
+ \| \hat{A} ^* \psi ^- \| ^2 = \| \xi _V ^* \| ^2  + \| \hat{A} ^* \psi ^- \| ^2 , \ee so that $\| \hat{A} ^* \psi ^- \|
= \| \psi ^- \|$. It follows that we can choose $U \in \mc{U} (n)$ such that $\hat{U} ^* \psi ^- = \hat{A} ^* \psi ^-$,
and that with this choice of $U$, $V(U)$ is a unitary extension of $V$ with $z=1$ as an eigenvalue.
\end{proof}

\begin{thm}
    The simple symmetric linear transformation $\mu ^{-1} ( Z' _\Theta ) $ will be a densely defined
simple symmetric operator if and only if the limit of $(1 - z)  U ( U -\Theta (z) ) ^{-1}$ as $z$
approaches $1$ non-tangentially vanishes for all $U \in \mc{U} (n)$. This happens if and only if
the angular derivative of $\Theta \vec{x}$ at $z=1$ does not exist for any $\vec{x} \in \bm{C} ^n$.

\label{thm:dense}
\end{thm}

\begin{proof}
    This is an immediate consequence of the previous lemma and Proposition \ref{prop:evalues}.
\end{proof}

\subsection{Regular simple symmetric operators with deficiency indices $(n,n)$}

    Now suppose that $B$ is a simple symmetric linear transformation on $\mc{H}$
with deficiency indices $(n,n)$. Such a linear transformation is called regular if
$B-z$ is bounded below for all $z \in \bm{C}$. The isometric linear transformation
$V = \mu (B)$ is called the Cayley
transform of $B$. This $V$ is a simple regular isometric linear transformation with deficiency
indices $(n,n)$. Here an isometric linear transformation is called regular if
$V-z$ is bounded below for all $z \in \bm{C} \sm \{ 1 \}$. The fact that $V$ is regular,
and the results of Section \ref{subsection:spectra} show that $\Theta _V$ is inner
and analytic on some neighbourhood of any given point $z \in \bm{T} \sm \{ 1 \}$.

Let $\Theta := \Theta _V$. Since $\Theta _V (0) = 0$, $(Z_\Theta  ' ) ^* $ is unitarily equivalent to $ (X  _ \Theta ) '$,
the isometric linear transformation which
acts as multiplication by $1/z$ on the orthogonal complement of the $n-$dimensional
subspace spanned by the vectors $\{ \delta _0 ^{\vec{x} } | \  \vec{x} \in \bm{C} ^n \} \subset K^2 _\Theta$ of point evaluations
at zero. Since $\Theta _V (0) = 0$, these are the constant functions in $K^2 _\Theta$. Let $\Phi = \Theta \circ \mu$,
and let $M $ be the self-adjoint operator of multiplication by the independent variable in $L^2 _n (\bm{R})$.
Then the image of $(X _\Theta ) ' $ under the canonical unitary map $\mc{U}$ of $K^2 _\Theta$ onto
$K^2 _\Phi \subset H^2 _n (\bm{U} )$ is the isometric linear transformation $\mu ^* (M ) _{\Phi} ' $ which acts
as multiplication by $\mu ^* (z) = \ov{\mu (\ov{z}) } = \frac{z+i}{z-i}$ on the domain of all functions in $K^2 _\Phi$ which vanish at $z=i$. Let
$M  _\Phi := (\mu ^*) ^{-1} ( \mu ^* (M) _\Phi ' ) $. Then $M _\Phi$ is a simple symmetric linear transformation
which acts as multiplication by the independent variable on its domain $\dom{M _\Phi } = \ran{\mu ^* (M) _\Phi ' - \bm{1}}$.

We will say that the inner function $\Phi _B = \Theta \circ \mu$ is the Lifschitz characteristic function
of $B$. Note that since $\Theta _V (0) = 0$, $\Phi _B (i) =0$. Combining these observations with Lifschitz' result,
Theorem \ref{thm:represent}, yields the following:

\begin{thm}
    A simple symmetric linear transformation $B$ with deficiency indices $(n,n)$ and characteristic
function $\Phi _B$ is regular if and only if $\Phi _B \in H^\infty _{\bm{M} _n} (\bm{U} )$ is an inner function which has an analytic extension to an open neighbourhood of any fixed $x \in \bm{R}$. In this case $B$ is unitarily equivalent to $M  _{\Phi _B}$ which acts as multiplication by
the independent variable on $\dom{M _{\Phi _B}} \subset K^2 _\Phi$.
\end{thm}

\subsubsection{Remark} The result stated above can be generalized to any simple symmetric linear transformation whose characteristic function
$\Phi _B$ is an extreme point, for in this case $H^2 _\Theta = L^2 _\Theta$ (where $\Theta = \Phi \circ \mu ^{-1}$ ),
and the canonical unitary transformations from $H^2 _\Theta$ onto $K^2 _\Theta $ and $K^2 _\Theta $ onto $K^2 _\Phi$
take $\mu ^{-1} (Z' _\Theta)$ onto $M _\Phi$.

\subsubsection{Remark} Theorem \ref{thm:dense} provides necessary and sufficient conditions on $\Phi$
for $M _\Phi$ to be a densely defined symmetric operator.

Now suppose that $\Phi = \Theta \circ \mu$ satisfies the conditions of Theorem \ref{thm:dense} so that $M _\Phi$
is densely defined, and that $\mr{sp} (\Phi ) := \mu ^{-1} \left( \mr{sp} (\Theta ) \right) \subset \{ \infty \}$, so that $M _\Phi$
is regular. Let $M _\Phi (U)$ be the image of $\mu ^{-1} (Z _\Theta (U))$ under the canonical unitary transformation
of $H^2 _\Theta $ onto $K^2 _\Phi$. Then the regularity of $M _\Phi$ implies that the spectrum of each $M _\Phi (U)$
is purely discrete with no finite accumulation point. Hence if $\sigma (M _\Phi (U) ) = \{ \la _i (U) \} \subset \bm{R}$, it
follows that there are vectors $\vec{x} _i (U) \in \bm{C} ^n$ such that the point evaluation vectors
$\{ \delta _{\la _i (U) } ^{\vec{x} _i (U)} \}$ form a total orthogonal set of eigenvectors to $M _\Phi (U)$ for
each $U \in \mc{U} (n)$, $M _\Phi (U) \delta _{\la _i (U) } ^{\vec{x} _i (U) } = \la _i (U) \delta _{\la _i (U) } ^{\vec{x} _i (U) }$.
Here the point evaluation vectors in $K^2 _\Phi$ have the form
\be \delta _\la ^{\vec{x}} (z) = \frac{i}{2\pi} \frac{\bm{1} - \Phi (z) \Phi ^* (\la )}{z - \ov{\la}} \vec{x} .\ee
Moreover if $M _\Phi$ is densely defined then each point evaluation vector $\delta _\la ^{\vec{x}}$ is an eigenvector
to $M _\Phi ^*$. Indeed, given any $f \in \dom{M _\Phi }$,
\be \ip{ M _\Phi f}{\delta _\la ^{\vec{x}} } _\Phi = \la (f (\la ) , \vec{x} ) = \ip{f}{\ov{\la} \delta _\la ^{\vec{x}}} _\Phi ,\ee
which shows that $\delta _\la ^{\vec{x}} \in \dom{ M _\Phi ^*}$ and that $M _\Phi ^* \delta _\la ^{\vec{x}} = \ov{\la} \delta _\la ^{\vec{x}}$.
Here, $\ip{\cdot}{\cdot} _\Phi$ denotes the inner product in $K^2 _\Phi$ (which is the usual $L^2$ inner product since we are
assuming $\Phi$ is inner).

In summary any regular simple symmetric linear transformation $B$ with deficiency indices $(n,n)$ is unitarily equivalent to multiplication by
the independent variable, $M _\Phi$ in a model subspace $K^2 _\Phi \subset H^2 _n (\bm{U})$, where $\Phi = \Phi _B$ is the
Lifschitz characteristic function of $B$. $\Phi \in H^\infty _{\bm{M} _n } (\bm{U} )$ is inner, and the fact that $B$
is regular implies that $\Phi $ has an analytic extension to some open neighbourhood of each $x \in \bm{R}$. The transformation
$B$ is densely defined if and only if  $\Theta :=
\Phi \circ \mu ^{-1}$ is such that the angular derivative of $\Theta \vec{x}$ at $z=1$ does not exist for any $\vec{x} \in \bm{C} ^n$.
In this case $K^2 _\Phi$ has a $\mc{U} (n)$ -parameter family of total orthogonal sets of point evaluation vectors $\{ \delta _{\la _i (U) } ^{\vec{x} _i (U)} \}$ which are eigenvectors to self-adjoint extensions $M _\Phi (U)$ of $ M _\Phi$ with eigenvalues $\la _i (U)$.
The spectra $\sigma (M _\Phi (U) ) = \{ \la _i (U) \}$ are purely discrete with no finite accumulation points. Moreover
each $\delta _\la  ^{\vec{x}}$ is an eigenvector of $M _\Phi ^*$ to eigenvalue $\la$.

\subsubsection{Remark} The above representation results for densely defined regular simple symmetric linear operators with deficiency indices
$(n,n)$ apply in particular to regular symmetric differential operators of any finite order, and to their self-adjoint extensions.

\subsection{A model for c.n.u. contractions with defect indices $(n,n)$}
\label{subsection:model}

    In this subsection we show that if $V$ is any partial isometry with finite and equal defect indices $(n,n)$, then
$V$ is unitarily equivalent to the partial isometry $Z _{\Theta _V} (0)$ acting in $L^2 _{\Theta _V} (\bm{T})$. If
$V' := V | _{\ker{V} ^{\perp}}$, an isometric linear transformation with deficiency indices $(n,n)$, then as shown in
\cite{Lifschitz2} (and reproduced in Theorem \ref{thm:represent} above),
$V'$ is unitarily equivalent to $Z _{\Theta _V} ' := Z _{\Theta _V} (0) | _{\ker{Z _\Theta (0)} ^\perp}$.
This establishes that the Lifschitz characteristic function of any isometric linear transformation $V'$ with indices $(n,n)$
is equal to the Nagy-Foias characteristic function of the partial isometric extension $V$ of $V'$ to the entire Hilbert space.
While natural, and known in the case where $\Theta _V$ is inner \cite{Foias}, for non-inner $\Theta _V$, this is not immediately obvious from the definitions of these two different characteristic functions.

To prove this, recall that $\Theta$, a contractive $\bm{M} _n$-valued analytic function on $\bm{D}$ is the characteristic
function of a partial isometry $V$ if and only if $\Theta (0) =0$. So to prove $V$ is isomorphic to $Z_{\Theta _V} (0)$ it suffices
to show that the characteristic function of $Z _{\Theta _V } (0)$ coincides with $\Theta _V$.

Let $T:= Z _{\Theta _V } (0)$, $Z:= Z_{\Theta _V}$ and let $P _+, P_-$ be the projectors onto
$\mf{D} _T$ and $\mf{D} _{T^*}$ respectively. Then the Nagy-Foias characteristic function of $T$ is \ba  \Theta _T (z ) & = & z P_- ( \bm{1} - zT^* ) ^{-1} P_+
= P_- \sum _{m=0} ^\infty z ^{m+1} (T^*) ^m P_+ \nonumber \\ & = & \sum _{m=0} ^\infty z^{m+1} P_- (Z ^{-1} - Z ^{-1} P_- ) ^m Z ^{-1} P_- Z .\ea

We will show that this coincides with $\Theta (z)$ by showing that if $\Theta (z) = \sum _{k=1} ^{\infty} c_k z^k$ that
$c_k = P_- (Z ^{-1} - Z ^{-1} P_- ) ^{k-1} Z ^{-1} P_- =: d_k$. As in Section \ref{subsection:ACmeas}, let $l_k := P_- Z^{-k} P_-$,
and let $P = P_-$. Then,

\ba d_k & = & P (Z ^{-1} - Z ^{-1} P ) ^{k-2} (Z^{-1} - Z^{-1} P)  Z^{-1} P   \nonumber \\
& = &  P (Z ^{-1} - Z ^{-1} P ) ^{k-2} Z^{-2} P - P  (Z ^{-1} - Z ^{-1} P ) ^{k-2} Z^{-1} P Z^{-1} P \nonumber \\
& & P (Z^{-1} - Z^{-1} P) ^{k-3} (Z^{-1} - Z^{-1} P) Z^{-2} P - d_{k-1} l_1 \nonumber \\
& = & P (Z^{-1} - Z^{-1} P) ^{k-3} Z^{-3} P - d_{k-2} l_2 - d_{k-1} l_1 \nonumber \\
& = & P (Z^{-1} - Z^{-1} P) Z^{-(k-1)} P -d_1 l_{k-1} - ... - d_{k-1} l_1 \nonumber \\
& = & l_k - d_1 l_{k-1} - ... - d_{k-1} l_1 .\ea

In the last line above the fact that $l_1 =d_1$ was used. It follows that $d_k = l _k  - \sum _{j=1} ^{k-1} d_j l_{k-j}$.
Since $d_1 = l_1 = c_1$, it follows from Lemma \ref{lemming:idone} that $d_k = c_k$. We have established the following:

\begin{prop}
    If $\Theta (0) = 0$ then the characteristic function of the partial isometry $Z_\Theta (0)$ coincides with $\Theta$.
Hence if $V$ is any partial isometry with finite defect indices $(n,n)$, $V$ is unitarily equivalent to $Z_{\Theta _V} (0)$.
\label{prop:charfunrel}
\end{prop}

\subsubsection{Remark} It follows that any completely non-unitary contraction $T$ with defect indices $(n,n)$ is unitarily equivalent to
some extension of the partial isometry $Z _{\Theta _{T(0)} } (0)$. More precisely, recall from Section \ref{subsection:Weylint} that
if $T$ is a contraction with defect indices $(n,n)$ on a Hilbert space $\mc{H}$, that $T_0 = T - TP_+$ where $P_+$ projects
onto $\mf{D} _T = \ran{D_T}$, $D_T = \sqrt{ \bm{1} - T^* T}$ is a partial isometry with defect indices $(n,n)$. Let $\Theta$
be a $\bm{M} _n$-valued matrix analytic function on $\bm{D}$ which coincides with the Nagy-Foias characteristic function
of $\Theta _{T_0}$, and let $U$ be the unitary transformation such that $U^* T_0 U = Z_{\Theta} (0)$. Then
$U^* T U = Z _\Theta (A)$ where $A \in \bm{M} _n$ has components given by $A _{ij} := \left( U^* T U 1/z e_i , 1/z e_j \right) _\Theta$.
This provides a model for any completely non-unitary contraction with defect indices $(n,n)$.

The following calculation helps to relate the Nagy-Foias characteristic function of $Z_\Theta (A)$ to that of
$Z_\Theta (0)$. The characteristic function $\Theta _T $ of $T := Z_\Theta (A)$ is, by definition,
\be \Theta _T (z) = \left( - T + z D_{T^*} (1 -z T^* )^{-1} D_T \right) | _{\mf{D} _T} .\ee
Now $D_{T^*} = D_{T^*} P_-$ and $D_T = P_+ D_T = Z^{-1} P_- Z $. So let $\Gamma (z) := z P_- (1 - Z_\Theta (A) ^* ) ^{-1}   Z^{-1} P_-$

\begin{prop}
    The matrix function $\Gamma (z) = \Theta (z) (\bm{1} - A^* \Theta (z) ) ^{-1}$.
    \label{prop:charfunrel2}
\end{prop}

\begin{proof}{ (Sketch)}
    This proof is very similar to previous calculations in Section \ref{subsection:ACmeas}.
As before let $\Theta (z) := \sum _{k=1} ^\infty c_k z^k$,
and let $\Gamma (z) = \sum _{k=1} ^\infty d_k z^k$.
\be \Gamma (z) = \sum _{m=0} ^\infty z^{m+1} P_- (Z_\Theta (A) ^*) ^m Z^{-1} P_-
= \sum _{m=0} ^\infty z^{m+1} P_- (Z ^{-1} + P_- (A^*-1) P_- Z^{-1}) ^m Z^{-1} P_- .\ee

Let $b_k$, $k \in \bm{N}$ be the coefficients of $\Theta (z) (\bm{1} - A^* \Theta (z) ) ^{-1}$ and $d_k$ be the coefficients
of $\Gamma (z)$. Now using the same methods as in Lemma \ref{lemming:idone} it is easy to calculate that
\be b_m = l_m + \sum _{j=1} ^{m-1} l_j (A^* -1) b_{m-j} .\ee

By the definition of the $d_j$ and Elliott's formula, Proposition \ref{prop:Elliott}, one can show, as in the proof of Proposition \ref{prop:acmeasure} that
\be d_j = c_j + \sum _{k=1} ^{j-1} c_k A^* d_{j-k}.\ee Finally, using these two formulas and the one relating the $l_k$
and $c_k$, one can use a combinatorial identity, as in the proof of Proposition \ref{prop:acmeasure} to show
that $d_j =b_j$.
\end{proof}

\end{document}